\documentclass{compositio}
\usepackage{amsmath,amsthm,amsfonts,amssymb,bm,bbm,graphicx,color,comment,enumerate}

\allowdisplaybreaks
\usepackage
{hyperref}
\hypersetup{colorlinks=true,citecolor=blue,linkcolor=blue,urlcolor=blue}

\newcommand{\lz}{\left(}
\newcommand{\pz}{\right)}

\renewcommand{\epsilon}{\varepsilon}
\newcommand{\C}{\mathbb{C}}
\newcommand{\Z}{\mathbb{Z}}
\newcommand{\sumstar}{\sideset{}{^{*}}\sum}

\newcommand{\R}{\mathbb{R}}

\newcommand{\lab}{\left\lvert }
\newcommand{\rab}{\right\rvert }

\newcommand{\M}{\mathcal{M}}

\newcommand{\re}{\mathrm{Re}}

\newcommand{\im}{\mathrm{Im}}

\newcommand{\bfrac}[2]{\lz\frac{#1}{#2}\pz}

\newcommand{\Go}{\Gamma_{o}}
\newcommand{\Ge}{\Gamma_{e}}
\newcommand{\res}{\mathrm{res}}

\renewcommand{\mod}[1]{\text{ (mod $#1$)}}
\newcommand{\odd}{\mathrm{\ odd}}

\newcommand{\even}{\text{ even}}
\renewcommand{\l}{\ell}
\renewcommand{\phi}{\varphi}

\newcommand{\leg}[2]{\left(\frac{#1}{#2}\right)}

\usepackage[utf8]{inputenc}

\newtheorem{theorem}{Theorem}
\newtheorem{lemma}[theorem]{Lemma}
\newtheorem{cor}[theorem]{Corollary}
\newtheorem{prop}[theorem]{Proposition}
\newtheorem{defin}[theorem]{Definition}
\newtheorem{conjecture}[theorem]{Conjecture}

\numberwithin{theorem}{section}
\numberwithin{equation}{section}

\begin{document}
	
	\title[The Ratios conjecture and multiple Dirichlet series]{The Ratios conjecture for real Dirichlet characters and multiple Dirichlet series}
	\author{Martin \v Cech}
	\email{martinxcech@gmail.com}
	\address{Concordia University, Department of Mathematics and Statistics, 1455 de Maisonneuve West, H3G 1M8, Montreal, Canada}
	\classification{11M26 (primary), 11M32 (secondary).}
	\keywords{Dirichlet L-functions, Ratios conjecture, multiple Dirichlet series, Dirichlet characters, functional equation}

	\begin{abstract}
		Conrey, Farmer and Zirnbauer introduced a recipe to find asymptotic formulas for the sum of ratios of products of shifted L-functions. These ratios conjectures are very powerful and can be used to determine many statistics of L-functions, including moments or statistics about the distribution of zeros.
		
		We consider the family of real Dirichlet characters, and use multiple Dirichlet series to prove the ratios conjectures with one shift in the numerator and denominator in some range of the shifts. This range can be improved by extending the family to include non-primitive characters. All of the results are conditional under the Generalized Riemann hypothesis.
		
		This extended range is good enough to enable us to compute an asymptotic formula for the sum of shifted logarithmic derivatives near the critical line. As an application, we compute the one-level density for test functions whose Fourier transform is supported in $\left(-2,2\right)$, including lower-order terms.
	\end{abstract}
	
	\maketitle
	
	% A table of contents should normally not be included
	\section{Introduction}
	
	In \cite{Mon}, Montgomery studied the pair correlation of zeros of the Riemann zeta function, and Dyson famously observed that the resulting density matches that of the pair correlation of eigenvalues in a GUE ensemble of random matrices. Since then, it is believed that random matrix theory can be used to model many statistics of L-functions, such as moments \cite{CFKRS}, or the distribution of low-lying zeros \cite{KaSa1}, \cite{KaSa2}.
	
	Farmer \cite{Far} conjectured that \begin{equation}\label{key}
		\int_{0}^T\frac{\zeta(s+\alpha)\zeta(1-s+\beta)}{\zeta(s+\gamma)\zeta(1-s+\delta)}dt\sim T\frac{(\alpha+\delta)(\beta+\gamma)}{(\alpha+\beta)(\gamma+\delta)}-T^{1-\alpha-\beta}\frac{(\delta-\beta)(\gamma-\alpha)}{(\alpha+\beta)(\gamma+\delta)},
	\end{equation}
	where $s=1/2+it$, and $\alpha,\beta,\gamma,\delta\asymp\frac{1}{\log T},$ and noticed that it has many implications including Montgomery's pair correlation conjecture.
	
	In \cite{CFZ1}, \cite{CFZ2}, based on the conjecture above, Conrey, Farmer and Zirnbauer came up with the ratios conjectures, which give a general recipe to predict asymptotic formulas for the sum of ratios of products of shifted L-functions. These are very powerful as they are able to predict many other local or global statistics, which agree with the predictions coming from random matrix theory (see \cite{CoSn} for some applications). The conjectured asymptotics are expected to hold with a very strong error term, so, unlike random matrix theory, they are also able to predict lower order terms.
	
	\smallskip
	
	We will study the family of quadratic Dirichlet L-functions with one shift in the numerator and in the denominator. For a fundamental discriminant $d$, we denote by $\chi_d$ the Kronecker symbol $\leg{d}{\cdot}$. In this case, the ratios conjecture has the form \begin{equation}\label{ratios conjecture}
		\begin{aligned}
			\sumstar_{d\leq X}&\frac{L(1/2+\alpha,\chi_d)}{L(1/2+\beta,\chi_d)}=\sumstar_{d\leq X}\Bigg(\frac{\zeta(1+2\alpha)}{\zeta(1+\alpha+\beta)}A_D(\alpha,\beta)\\
			&+\bfrac{\pi}{d}^{\alpha}\frac{\Gamma(1/4-\alpha/2)\zeta(1-2\alpha)}{\Gamma(1/4+\alpha/2)\zeta(1-\alpha+\beta)}A_D(-\alpha,\beta)\Bigg)+O(X^{1/2+\epsilon}),
		\end{aligned}
	\end{equation}
	where the star indicates that the sums run over fundamental discriminants, and
	\begin{equation}\label{key}
		A_D(\alpha,\beta)=\prod_p\lz1-\frac1{p^{1+\alpha+\beta}}\pz^{-1}\lz1-\frac1{(p+1)p^{1+2\alpha}}-\frac{1}{(p+1)p^{\alpha+\beta}}\pz.
	\end{equation} The original conjecture asserts that this formula holds uniformly in $\alpha,\beta$ with $\lvert\re(\alpha)\rvert <1/4$, $\frac1{\log X}\ll\re(\beta)<1/4$, and $\im(\alpha),\im(\beta)\ll X^{1-\epsilon}$. We, conditionally under the generalized Riemann hypothesis (GRH), prove the conjecture in a smaller range of the shifts.
	
	We are not aware of any other results towards the proof of this conjecture over number fields. See the recent work of Bui, Florea and Keating \cite{BFK2} for similar results over function fields, using a different method.
	
	We now briefly describe the recipe of Conrey, Farmer and Zirnbauer. The L-functions in the numerator are replaced by Dirichlet polynomials using the approximate functional equation, while those in the denominator are expanded into a full Dirichlet series. To obtain the main terms in the conjecture, the sums are completed, only the diagonal terms are retained, and they are replaced by their average over the family.
	
	Our strategy is to use Mellin inversion (or Perron's formula), which leads to an integral that involves a triple Dirichlet series $A(s,w,z)$. The main terms arise naturally from residues of $A(s,w,z)$, which gives another evidence for the validity of the heuristics of Conrey, Farmer and Zirnbauer. An advantage of our approach is that there are no non-diagonal terms to be bounded. The error term depends on how far we can meromorphically continue the triple Dirichlet series, which we obtain from showing that $A(s,w,z)$ satisfies some functional equations.

	\smallskip
	Let us now state our results, which are conditional under GRH. Theorem \ref{Theorem for fundamental discriminants} is weaker than Theorem \ref{Theorem for all characters}, its purpose is mainly to illustrate the method of the proof and to illustrate how the main terms arise from the residues of the multiple Dirichlet series. The difference in the results can be explained on the level of functional equations satisfied by the associated multiple Dirichlet series, which is explained in Section \ref{Section overview and heuristic}.
	
	Let \begin{equation}\label{key}
		R_D(X,\alpha,\beta;f)=\sumstar_{d\geq1}\frac{L(1/2+\alpha,\chi_d)}{L(1/2+\beta,\chi_d)}f(d/X),
	\end{equation}
	where the sum runs over positive fundamental discriminants. We also let $f(x)$ be a smooth, fast-decaying weight function.
	
	To simplify some of the formulas, we denote by $\Gamma_e(s)$ and $\Gamma_o(s)$ the ratios of gamma factors that appear in the functional equation for even or odd characters, so that
	\begin{equation}\label{key}
		\Ge(s)=\frac{\Gamma\bfrac{1-s}{2}}{\Gamma\bfrac s2}, \hspace{20pt}
		\Go(s)=\frac{\Gamma\bfrac{2-s}{2}}{\Gamma\bfrac {s+1}2}.
	\end{equation} We denote by $\M f(s)$ the Mellin transform of $f(s)$, defined by \eqref{eqn: Mellin transform def}.
	
	\begin{theorem}\label{Theorem for fundamental discriminants}
		Assume GRH, and let $0<\lvert \re(\alpha)\rvert<\re(\beta)<1/2$. Then
		\begin{equation}\label{Result all characters}
			\begin{aligned}
				R_D&(X,\alpha,\beta;f)=\frac{X\M f(1)\zeta(1+2\alpha)}{2\zeta(2)\zeta(1+\alpha+\beta)}P_D(1/2+\beta,1/2+\alpha)\\
				&+\frac{X^{1-\alpha}\M f(1-\alpha)\zeta(1-2\alpha)\pi^\alpha\Ge(1/2+\alpha)}{2\zeta(2)\zeta(1-\alpha+\beta)}P_D(1/2+\beta,1/2-\alpha)\\
				&+O_{\alpha,\beta}(X^{1-\re(\alpha)/2-\re(\beta)/2+\epsilon}),
			\end{aligned}
		\end{equation}
		where 
		\begin{equation}\label{key}
			P_D(z,w)=\prod_p\lz1+\frac{1-p^{z-w}}{(p^{z+w}-1)(p+1)}\pz.
		\end{equation}
	\end{theorem}
	The error term in this result is not uniform in the shifts $\alpha$, $\beta$. It is possible to obtain a uniform result using a similar estimate as in Section \ref{Section bound in vertical strips}, which would introduce a factor of size $\lvert \alpha\rvert ^{c+\epsilon}\lvert \beta\rvert ^{\epsilon}$ for some $c>0$ into the error term, limiting the size of the imaginary parts. Note also that this theorem doesn't hold in the range $\alpha=\beta,$ which is important for some applications. Both of these aspects can be improved by making the sum run over all quadratic characters, not only the primitive ones, as shown in Theorem \ref{Theorem for all characters}.
	
	The main terms from \eqref{Result all characters} agree with those in \eqref{ratios conjecture} (adjusted for the smooth weights). We remark that the computations leading to them are the same as those in the recipe of Conrey, Farmer and Zirnbauer, but they come from a different source. While in the heuristic, one obtains the main terms by discarding the non-diagonal terms assuming that they only contribute into the error, in our case they come from the residues of certain triple Dirichlet series.
	
	\medskip
	
	It turns out that we can get a better range of the shifts if we extend our family to contain all characters, including non-primitive ones. The reason is that in this case, the associated triple Dirichlet series has an extra functional equation. We will change notation from Section \ref{Section series for all characters} on and denote by $\chi_n$ the Jacobi symbol $\leg{\cdot}{n}$. We denote by $L_{(2)}(s,\chi_n)$ the L-function with the Euler factor at 2 removed. Then we get the following result:
	
	\begin{theorem}\label{Theorem for all characters}
		Assume GRH and let $\epsilon>0$. Then for $\re(\alpha)>0$, $\re(\beta)>\epsilon$, $1+\re(\beta)>\re(\alpha)$, we have
		\begin{equation}\label{Asymptotic for ratios of all characters}
			\begin{aligned}
				&\sum_{n\odd}\frac{L_{(2)}(1/2+\alpha,\chi_n)}{L_{(2)}(1/2+\beta,\chi_n)}f(n/X)\\
				&=X\M f(1)\frac{\zeta_{(2)}(1+2\alpha)}{2\zeta_{(2)}(1+\alpha+\beta)}\prod_{p>2}\lz1+\frac{p^{\alpha-\beta}-1}{p^{1+\alpha-\beta}(p^{1+\alpha+\beta}-1)}\pz\\
				&+X^{1-\alpha}\M f(1-\alpha)\pi^{\alpha}\lz\Go\lz\frac12+\alpha\pz+\Ge\lz\frac12+\alpha\pz\pz\frac{\zeta(1-2\alpha)P\lz\frac32-\alpha+\beta\pz}{\zeta(2)\zeta(1-\alpha+\beta)(6-2^{\alpha-\beta+1})}\\
				&+O\lz(1+\lvert \alpha\rvert )^\epsilon\lvert \beta\rvert ^\epsilon X^{N(\alpha,\beta)+\epsilon}\pz,
			\end{aligned}
		\end{equation} where
		\begin{equation}\label{key}
			P(z)=\prod_{p}\lz1+\frac{1}{(p^{z-1/2}-1)(p+1)}\pz,
		\end{equation} and 
		\begin{equation}\label{key}
			N(\alpha,\beta)=\max\left\{1-2\re(\alpha),1-2\re(\beta),\frac12-\re(\alpha),\frac12-\re(\beta),-\frac52\right\}.
		\end{equation}
	\end{theorem}
	The condition $1+\re(\beta)>\re(\alpha)$ is to ensure convergence of the product $P(3/2-\alpha+\beta)$, and the $-5/2$ in the error term comes from our definition of $\tilde S_j$ in Section \ref{Section bound in vertical strips}.
	
	In this case, the error term is uniform in $\alpha,\beta$. Let us emphasize that the imaginary parts may grow as fast as any power of $X$, which is better than the original conjecture.
	
	This result also allows us to differentiate with respect to $\alpha$ and take $\alpha=\beta=r$, thus obtain
	
	\begin{theorem}\label{Theorem for log derivatives}
		Assume GRH and Let $\re(r)>\epsilon$. Then
		\begin{equation}\label{Asymptotic formula for ratios of log derivative}
			\begin{aligned}
				\sum_{\substack{n\geq1,\\n\odd}}&\frac{L'(1/2+r,\chi_n)}{L(1/2+r,\chi_n)}f(n/X)=\frac{X\M f(1)}{2}\lz\frac{\zeta'(1+2r)}{\zeta(1+2r)}+\sum_{p>2}\frac{\log p}{p(p^{1+2r}-1)}\pz\\
				&-X^{1-r}\M f(1-r)\pi^r\lz\Go\lz\frac12+r\pz+\Ge\lz\frac12+r\pz\pz\frac{\zeta(1-2r)}{4}\\
				&+O\lz1+ \lvert r\rvert ^{\epsilon}X^{N(r)+\epsilon}\pz,
			\end{aligned}
		\end{equation} where
		\begin{equation}\label{key}
			N(r)=\max\{1-2\re(r),1/2-\re(r),-5/2\}.
		\end{equation}
	\end{theorem}
	At this point, we can sieve out the non-primitive characters and obtain an asymptotic formula for the sum running over square-free integers:
	\begin{theorem}\label{Theorem for log derivatives for squarefree moduli}
		Assume GRH and let $\epsilon<\re(r)<1/4$. Then
		\begin{equation}
			\begin{aligned}
				&\sum_{\substack{n\geq1,\\n\odd}}\frac{\mu^2(n)L'(1/2+r,\chi_n)}{L(1/2+r,\chi_n)}f(n/X)\\	
				&=\frac{2X\M f(1)}{3\zeta(2)}\lz\frac{\zeta'(1+2r)}{\zeta(1+2r)}+\sum_{p>2}\frac{\log p}{(p+1)(p^{1+2r}-1)}\pz\\
				&-X^{1-r}\M f(1-r)\pi^r\lz\Go\lz\frac12+r\pz+\Ge\lz\frac12+r\pz\pz\frac{\zeta(1-2r)}{4\zeta_{(2)}(2-2r)}\\
				&+O\lz \lvert r\rvert ^\epsilon X^{1-2\re(r)+\epsilon}\pz.
			\end{aligned}
		\end{equation}
	\end{theorem}
	As opposed to Theorems \ref{Theorem for all characters} and \ref{Theorem for log derivatives}, but similarly as Theorem \ref{Theorem for fundamental discriminants}, we now only have primitive characters, so we can compare the result with the prediction. This is done in Appendix B, where we show that the main terms in Theorem \ref{Theorem for log derivatives for squarefree moduli} agree with those coming from the recipe. However, in this case, the computations leading to Theorem \ref{Theorem for log derivatives for squarefree moduli} are different from those in the heuristic.
	
	\smallskip
	
	As an application of our results, we compute the one-level density in our family of quadratic Dirichlet characters. For an even Schwartz function $h(x)$ whose Fourier transform $\hat h $ is supported in $[-a,a]$ for some $a>0,$ the one-level density is defined by
	\begin{equation}\label{key}
		D(X;h)=\frac{1}{F(X)}\sum_{\substack{n\geq1\\n\odd}}\mu^2(n)f\bfrac nX\sum_{\gamma_n}h\bfrac{\gamma_n\log X}{2\pi},
	\end{equation} where $\gamma_n$ runs over the imaginary parts of the non-trivial zeros of $L(s,\chi_n)$, and
	\begin{equation}\label{key}
		F(X)=\sum_{\substack{n\geq 1,\\n\odd}}\mu^2(n)f\bfrac{n}{X}.
	\end{equation}
	By the conjectures of Katz and Sarnak \cite{KaSa1}, \cite{KaSa2}, \begin{equation}\label{key}
		D(X;h)\sim\int_{-\infty}^{\infty}h(u)\lz1-\frac{\sin(2\pi u)}{2\pi u}\pz du,
	\end{equation} where the kernel of the integral governs the distribution of eigenvalues close to 1 in a symplectic ensemble of random matrices. This has been proved for $a<1$ and for $a<2$ under GRH by Özlük and Snyder \cite{OzSn}, and increasing the support further is a notoriously difficult problem. The ratios conjecture \eqref{ratios conjecture} implies that the above asymptotic holds for arbitrarily large $a$, and it allows us to compute all lower order terms up to square-root error. In \cite{Mil1}, Miller showed unconditionally that for limited $a$, the ratios conjecture prediction agrees with reality, including lower order terms. See also \cite{Mil2}, \cite{MiMo}, \cite{GJM+}, \cite{HMM} for similar results in other families, \cite{BFK1} for related work over function fields, or \cite{DHJ}, where the ratios conjecture is applied to compute the one-level density in two families of elliptic curves. 
	
	On the other hand, Fiorilli and Miller \cite{FiMi} computed the one-level density in the family of all Dirichlet characters modulo $q$, including lower order terms beyond square-root. They discovered a term not predicted by the ratios conjecture, concluding that the conjectured error is essentially best possible.
	
	As a consequence of Theorem \ref{Theorem for log derivatives for squarefree moduli}, we compute the one-level density provided $a<2$, thus recovering the results of Özlük and Snyder, including lower order terms with a power-saving error.
	\begin{cor}\label{Corollary one-level density} Assume GRH and let $h$ be as above. Then
		\begin{equation}\label{key}
			\begin{aligned}
				&\sum_{\substack{n\geq1,\\n\odd}}\mu^2(n)f\bfrac nX\sum_{\gamma_n}h\bfrac{\gamma_n\log X}{2\pi}\\
				&=\frac{2X}{\log X}\int_{-\infty}^{\infty}h(u)\Bigg\{\frac{2\M f(1)}{3\zeta(2)}\lz\frac{\zeta'}{\zeta}\lz1+\frac{4\pi i u}{\log X}\pz+\sum_{p>2}\frac{\log p}{(p+1)(p^{1+4\pi iu/\log X}-1)}\pz\\	
				&-\frac{e^{-2\pi i u}\M f\lz1-\frac{2\pi i u}{\log X}\pz\pi^{\frac{2\pi i u}{\log X}}\lz\Go\lz\frac{1}{2}+\frac{2\pi i u}{\log X}\pz+\Ge\lz\frac{1}{2}+\frac{2\pi i u}{\log X}\pz\pz\zeta\lz1-\frac{4\pi i u}{\log X}\pz}{4\zeta_{(2)}\lz2-\frac{4\pi i u}{\log X}\pz}
				\Bigg\}du\\
				&-\frac{1}{\log X}\sum_{\substack{n\geq 1,\\n\odd}}\mu^2(n)f\bfrac nX\log\bfrac \pi n\int_{-\infty}^{\infty}h(u) du\\
				&-\frac{X\M f(1)}{3\zeta(2)\log X}\int_{-\infty}^{\infty}h(u)\lz\frac{\Go'}{\Go}\lz\frac{1}{2}-\frac{2\pi iu}{\log X}\pz+\frac{\Ge'}{\Ge}\lz\frac{1}{2}-\frac{2\pi iu}{\log X}\pz\pz du\\
				&+O\lz X^{1/2+a/4+\epsilon}\pz.
			\end{aligned}
		\end{equation}
		This implies that if $a<2$, \begin{equation}\label{key}
			D(X;h)=\int_{-\infty}^{\infty}h(u)\lz1-\frac{\sin(2\pi u)}{2\pi u}\pz du+O\bfrac 1{\log X}.
		\end{equation}
	\end{cor}
	The proof is given in Section \ref{Section one-level density}, where we also explain that any improvement of the error term in Theorem \ref{Theorem for log derivatives for squarefree moduli}, which is directly related to a meromorphic continuation of certain triple Dirichlet series, would allow us to increase the support.
	\smallskip
	
	Our strategy to prove Theorems \ref{Theorem for fundamental discriminants} and \ref{Theorem for all characters} is to rewrite the sums as integrals using Mellin inversion (a smooth version of Perron's formula), and then investigate the analytic properties of the relevant triple Dirichlet series. We show that they have two poles, whose residues give rise to the two main terms. The error then depends on how far our triple Dirichlet series can be meromorphically extended, which in turn depends on whether they satisfy certain functional equations. 
	
	A general theory of multiple Dirichlet series has been developed by Bump, Chinta, Diaconu, Friedberg, Hoffstein and others (see for example \cite {BFG}, \cite{Blo}, \cite{DGH}, \cite{GH}, or the expository papers \cite{Bum}, \cite{BFH}, \cite{CFH} for an introduction to the theory). In \cite{DGH}, the authors prove that if certain multiple Dirichlet series admit a meromorphic continuation beyond a certain point, then the moments of the Riemann zeta function and quadratic Dirichlet L-functions satisfy the asymptotics predicted by random matrix theory. 
	
	According to the heuristics developed in \cite{BFH} and \cite{CFH} that we present in Section \ref{Section overview and heuristic}, our triple Dirichlet series are expected to satisfy two functional equations. To prove these rigorously, one usually twists the L-functions in the coefficients of the triple Dirichlet series by certain weights, carefully chosen so that the equations are satisfied. For more details, see \cite{CFH} for an exposition of the general theory, or Blomer's work \cite{Blo}, where the situation is ``simple'' enough, so the author gives very explicit results and computations.
	
	We use a different method in this paper. In the case of fundamental discriminants, we only have one functional equation, which is essentially due to the fact that
	\begin{equation}\label{key}
		L_D(s,\chi)=\sumstar_{d\geq1}\frac{\chi(d)}{d^s}
	\end{equation} doesn't satisfy any relation between $L_D(s,\chi)$ and $L_D(1-s,\chi)$. This explains why we obtain a weaker result in this case.
	
	When we include the non-primitive characters, we are summing over a nicer set, so both functional equations potentially hold. However, we have to deal with the L-functions of non-primitive characters -- this is usually done by inserting the extra weights mentioned above. Instead, we introduce a functional equation that is valid for all Dirichlet characters, but where on the other side, one has a different Dirichlet series whose coefficients are twisted Gauss sums (see Proposition \ref{Functional equation with Gauss sums} and Appendix A). It then becomes straightforward to prove the functional equations for our triple Dirichlet series, but with a different triple Dirichlet series on the other side.
	
	\section{Preliminaries}
	
	Throughout the paper, $\epsilon$ will denote a small positive number, that may be different at each appearance. All implied constants can depend on $\epsilon.$
	
	For an odd positive integer $n$, $\chi_n$ denotes the quadratic Dirichlet character given by the Jacobi symbol $\leg{\cdot}{n}$. The character $\chi_n$ is primitive for square-free $n$, and it is even for $n\equiv1\mod 4$ and odd for $n\equiv3\mod 4$. We will also work with the Kronecker symbol $\leg{k}{\cdot},$ which is periodic modulo $k$ only if $k\equiv0,1\mod 4$. However, we often use the fact that if $n$ is odd, then $\leg{k}{n}=\leg{4k}{n},$ and $\leg{4k}{\cdot}$ is a Dirichlet character modulo $4k$ for any odd $k$. We also denote by $\psi_1,\psi_{-1}$ the principal and non-principal characters modulo $4$ respectively, and by $\psi_2,\psi_{-2}$ the two primitive quadratic characters modulo 8 given by the Kronecker symbols $\psi_j(n)=\leg{4j}{n},$ and by $\psi_0$ the primitive principal character, so that $\psi_0(n)=1$ for all $n\in\Z$.
	
	In sections \ref{Section overview and heuristic}, \ref{Section series for fund discr} and \ref{Section Proof of Theorem 1}, $\chi_d$ denotes the Kronecker symbol $\leg{d}{\cdot}$ for a fundamental discriminant $d$.
	
	For a Dirichlet character $\chi$ modulo $n$, we define the shifted Gauss sums \begin{equation}\label{key}
		\tau(\chi,q)=\sum_{j\mod n}\chi(j)e(jq/n), 
	\end{equation} where we use the standard notation $e(x)=e^{2\pi ix}$.
	
	If $\chi$ is a primitive character, then 
	\begin{equation}\label{Gauss sums for primitive characters}
		\tau(\chi,q)=\bar\chi(q)\tau(\chi,1).
	\end{equation}
	\begin{lemma}\label{Multiplicativity of Gauss sums}
		Let $\chi_1,\chi_2$ be two Dirichlet characters modulo $n_1$ and $n_2$, respectively, and assume that $(n_1,n_2)=1$. Then for $\chi_1\chi_2$ considered as a Dirichlet character modulo $n_1n_2,$ we have \begin{equation}\label{key}
			\tau(\chi_1\chi_2,q)=\chi_1(n_2)\chi_2(n_1)\tau(\chi_1,q)\tau(\chi_2,q).
		\end{equation}
	\end{lemma}
	\begin{proof}
		By the Chinese remainder theorem, we can write any $j\in\Z/(n_1n_2Z)$ uniquely as  \begin{equation}\label{key}
			j=j_1n_2\overline {n_2}+j_2n_1\overline {n_1},
		\end{equation} where $n_1\overline{n_1}\equiv1\mod{n_2}$, $n_2\overline{n_2}\equiv1\mod {n_1},$ $j_1\in\Z/n_1\Z$, and $j_2\in\Z/n_2\Z$. Hence, we have \begin{equation}\label{key}
			\begin{aligned}
				\tau(\chi_1\chi_2,q)&=\sum_{j\mod{n_1n_2}}\chi_1(j)\chi_2(j)e\bfrac{qj}{n_1n_2}\\
				&=\sum_{j_1\mod{n_1}}\sum_{j_2\mod{n_2}}\chi_1(j_1)\chi_2(j_2)e\bfrac{qj_1\overline{n_2}}{n_1}e\bfrac{qj_2\overline{n_1}}{n_2}\\
				&%=\tau(\chi_1,q\overline{n_2})\tau(\chi_2,q\overline{n_1})
				=\chi_1(n_2)\chi_2(n_1)\tau(\chi_1,q)\tau(\chi_2,q).
			\end{aligned}
		\end{equation}
	\end{proof}
	
	\begin{lemma}\label{Lemma changing Gauss sums}
		\begin{enumerate}
			\item If $\l\equiv1\mod 4$, then \begin{equation}\label{key}
				\tau\lz\bfrac{4\l}{\cdot},q\pz=\begin{cases}
					0,&\hbox{if $q$ is odd,}\\
					-2\tau\lz\leg{\cdot}{\l},q\pz,&\hbox{if $q\equiv2\mod 4$,}\\
					2\tau\lz\leg{\cdot}{\l},q\pz,&\hbox{if $q\equiv0\mod 4$.}
				\end{cases}
			\end{equation}
			\item If $\l\equiv3\mod 4$, then
			\begin{equation}\label{key}
				\tau\lz\bfrac{4\l}{\cdot},q\pz=\begin{cases}
					0,&\hbox{if $q$ is even,}\\
					-2i\tau\lz\leg{\cdot}{\l},q\pz,&\hbox{if $q\equiv1\mod 4$,}\\
					2i\tau\lz\leg{\cdot}{\l},q\pz,&\hbox{if $q\equiv3\mod 4$.}
				\end{cases}
			\end{equation}
		\end{enumerate}
	\end{lemma}
	\begin{proof}
		Quadratic reciprocity gives for $\l\equiv1\mod 4$ \begin{equation}\label{key}
			\tau\lz\leg{4\l}{\cdot},q\pz=\tau\lz\leg{\cdot}{\l}\psi_1,q\pz,
		\end{equation} and for $\l\equiv3\mod 4$ \begin{equation}\label{key}
			\tau\lz\leg{4\l}{\cdot},q\pz=\tau\lz\leg{\cdot}{\l}\psi_{-1},q\pz.
		\end{equation} The result then follows from Lemma \ref{Multiplicativity of Gauss sums} with $n_1=\l$ and $n_2=4$.
	\end{proof}
	
	For quadratic characters, we also define $G\lz\chi_n,q\pz$ by \begin{equation}\label{key}
		\begin{aligned}
			G\lz\chi_n,q\pz&=\lz\frac{1-i}{2}+\leg{-1}{n}\frac{1+i}{2}\pz\tau\lz\chi_n,q\pz=\\
			&=\begin{cases}
				\tau\lz\chi_n,q\pz,&\hbox{if $n\equiv1\mod 4$,}\\
				-i\tau\lz\chi_n,q\pz,&\hbox{if $n\equiv3\mod 4$,}
			\end{cases}
		\end{aligned}
	\end{equation} whose advantage is that they are multiplicative in $n$, so that for $(m,n)=1$, we have\begin{equation}\label{key}
		G(\chi_{mn},q)=G(\chi_m,q)G(\chi_n,q).
	\end{equation}
	
	Moreover, if $p$ is an odd prime and $p^a\mid \mid q$, we have by \cite[Lemma 2.3]{Sou} \begin{equation}\label{Sound's Gauss sums - exact formula}
		G\lz\leg{\cdot}{p^k},q\pz=\begin{cases}\phi(p^k),&\hbox{if $k\leq a$, $k$ even,}\\
			0,&\hbox{if $k\leq a$, $k$ odd,}\\
			-p^a,&\hbox{if $k=a+1$, $k$ even,}\\
			\leg{qp^{-a}}{p}p^{a}\sqrt p,&\hbox{if $k=a+1$, $k$ odd,}\\
			0,&\hbox{if $k\geq a+2$,}\\
		\end{cases}
	\end{equation}
	
	For primitive Dirichlet characters of conductor $d$, we have
	\begin{equation}\label{Functional equation for L-functions of primitive characters}
		L(s,\chi)=\epsilon(\chi)\bfrac{\pi}{d}^{s-1/2}\Gamma_{e/o}(s)L(1-s,\bar\chi),
	\end{equation} where $\epsilon(\chi)=\frac{a\tau(\chi,1)}{\sqrt d}$, and $a=1$ and $\Gamma_{e/o}=\Gamma_e$ if $\chi$ is even, or $a=-i$ and $\Gamma_{e/o}=\Gamma_o$ if $\chi$ is odd.
	
	We now introduce a functional equation valid for all Dirichlet characters $\chi$ modulo $n$. 
	\begin{prop}\label{Functional equation with Gauss sums}
		Let $\chi$ be any character modulo $n$. Then we have \begin{equation}\label{Equation functional equation with Gauss sums}
			L(s,\chi)=\epsilon(\chi)\frac{\pi^{s-1/2}}{n^s}\Gamma_{e/o}(s) K(1-s,\chi),
		\end{equation} where \begin{equation}\label{key}
			K(s,\chi)=\sum_{q=1}^\infty\frac{\tau(\chi,q)}{q^s},
		\end{equation} \begin{equation}\label{key}
			\Gamma_{e/o}(s)=\begin{cases}
				\Gamma_{e}(s),&\hbox{if $\chi$ is even, or}\\
				\Gamma_{o}(s),&\hbox{if $\chi$ is odd,}
			\end{cases}
		\end{equation} and
		\begin{equation}\label{key}
			\epsilon(\chi)=\begin{cases}
				1,&\hbox{if $\chi$ is even, or}\\
				-i,&\hbox{if $\chi$ is odd.}
			\end{cases}
		\end{equation}
	\end{prop}
	Note that if $\chi$ is a primitive character, then we can use \eqref{Gauss sums for primitive characters} and recover \eqref{Functional equation for L-functions of primitive characters} from \eqref{Equation functional equation with Gauss sums}.
	\begin{proof}
		Follow one of the usual proofs of the functional equation that uses Poisson summation. The application of the Poisson summation leads to some Gauss sums, and for a primitive character $\chi$, one uses \eqref{Gauss sums for primitive characters} to change the Gauss sums back to characters. Skipping this last step and leaving the Gauss sums unchanged gives the proof.
		
		We include the details in Appendix A.
	\end{proof}
	
	We denote by $L_{(k)}(s,\chi)$ the L-function with Euler factors of $p\mid k$ removed, so
	\begin{equation}\label{key}
		L_{(k)}(s,\chi)=L(s,\chi)\prod_{p\mid k}\lz1-\frac{\chi(p)}{p^s}\pz
	\end{equation} 
	
	We now record some useful estimates that hold under GRH.
	First is the Lindel\"of bound: for $\re(s)\geq1/2$, \begin{equation}\label{Lindelof}
		\lvert L(s,\chi_n)\rvert \ll \lvert sn\rvert ^\epsilon.
	\end{equation}
	Next, if $n$ is squarefree so that $\chi_n$ is primitive, and for $\re(r)>\epsilon$, we also have\begin{equation}\label{Size of log derivative}
		\lab\frac{L'(1/2+r,\chi_n)}{L(1/2+r,\chi_n)}\rab\ll\log^2(\lvert r\rvert n).
	\end{equation}
	Finally, for any $n$ and $\re(s)>1/2+\epsilon$, we have
	\begin{equation}\label{1/L under GRH}
		\frac1{\lvert L(s,\chi_n)\rvert }\ll \lvert sn\rvert ^{\epsilon}.
	\end{equation}	
	
	Let \begin{equation}\label{key}
		L_D(s,\chi)=\sumstar_{d\geq 1}\frac{\chi(d)}{d^s}
	\end{equation} be the Dirichlet series where the sum runs over fundamental discriminants.
	\begin{lemma}\label{lemma - Dirichlet series of fund discr}
		For a Dirichlet character $\chi,$ we have \begin{equation}\label{key}
			L_D(s,\chi)=\lz\frac12+\frac{\chi(4)}{2\cdot4^s}+\frac{\chi(8)}{8^s}\pz\frac{L(s,\chi\psi_1)}{L(2s,\chi^2\psi_1)}+\lz\frac12-\frac{\chi(4)}{2\cdot 4^s}\pz\frac{L(s,\chi\psi_{-1})}{L(2s,\chi^2\psi_1)}.
		\end{equation}
	\end{lemma}	
	\begin{proof}
		Any integer $d$ is a fundamental discriminant if either \begin{enumerate}
			\item[a)] $d\equiv1\mod 4$ and is squarefree, or
			\item[b)] $d=4m$ with $m\equiv3\mod 4$ and $m$ is squarefree, or
			\item[c)] $d=4m$ with $m\equiv2\mod 4$ and $m$ is squarefree.
		\end{enumerate}
		We will compute the contribution of each of these three sets separately.
		
		For part a), we have:\begin{equation}\label{key}
			\begin{aligned}
				\sum_{\substack{d\equiv1\mod4}}\frac{\chi(d)\mu^2(d)}{d^s}&=\frac12\sum_{d\geq1}\frac{\chi(d)\mu^2(d)(\psi_1(d)+\psi_{-1}(d))}{d^s}\\
				&=\frac12\prod_{p}\lz1+\frac{\chi(p)\psi_1(p)}{p^s}\pz+\frac12\prod_p\lz1+\frac{\chi(p)\psi_{-1}(p)}{p^s}\pz\\
				%&=\frac12\prod_p\frac{1-\frac{\chi^2(p)\psi_1^2(p)}{p^{2s}}}{1-\frac{\chi(p)\psi_1(p)}{p^s}}+\frac12\prod_p\frac{1-\frac{\chi^2(p)\psi_{-1}^2(p)}{p^{2s}}}{1-\frac{\chi(p)\psi_{-1}(p)}{p^s}}=\\
				&=\frac12\lz\frac{L(s,\chi\psi_1)}{L(2s,\chi^2\psi_1)}+\frac{L(s,\chi\psi_{-1})}{L(2s,\chi^2\psi_1)}\pz.
			\end{aligned}
		\end{equation}
		
		Part b) gives:
		
		\begin{equation}\label{key}
			\begin{aligned}
				\sum_{\substack{d=4m,\\ m\equiv3\mod 4}}\frac{\chi(d)\mu^2(m)}{d^s}&=\frac{\chi(4)}{2\cdot4^s}\sum_{m\geq1}\frac{\chi(m)\mu^2(m)(\psi_1(m)-\psi_{-1}(m))}{m^s}\\
				&= \frac{\chi(4)}{2\cdot 4^s}\lz\frac{L(s,\chi\psi_1)}{L(2s,\chi^2\psi_1)}-\frac{L(s,\chi\psi_{-1})}{L(2s,\chi^2\psi_1)}\pz.
			\end{aligned}
		\end{equation}
		
		The condition in part c) is equivalent to $d=8m$ with $m$ odd and squarefree. Hence we obtain 
		\begin{equation}\label{key}
			\begin{aligned}
				\sum_{\substack{d=8m,\\ m\odd}}\frac{\chi(d)\mu^2(m)}{d^s}&=\frac{\chi(8)}{8^s}\sum_{m\geq1}\frac{\chi(m)\psi_1(m)\mu^2(m)}{m^s}\\
				&=\frac{\chi(8)}{8^s}\frac{L(s,\chi\psi_1)}{L(2s,\chi^2\psi_1)}.
			\end{aligned}
		\end{equation}
		
		Adding the three parts gives the result.
	\end{proof}

	\medskip
	
	For a function $f(x)$, we denote by $\M f(s)$ its Mellin transform defined by 
	\begin{equation}\label{eqn: Mellin transform def}
		\M f(s)=\int_{0}^\infty f(x)x^{s-1}dx,
	\end{equation} for these $s$ where the integral converges. If $f(x)$ is smooth, $\M f(s)$ decays faster than any polynomial in vertical strips.
	
	If $\M f(s)$ is analytic for $a<\re(s)<b$, then the inverse Mellin transform is given by
	\begin{equation}\label{key}
		f(x)=\frac1{2\pi i}\int_{(c)} x^{-s}\M f(s)ds,
	\end{equation} where the integral is taken over a vertical line $\re(s)=c$ and $a<c<b$ is arbitrary.
	
	The following estimate is a consequence of Stirling's formula: for a fixed $\sigma\in\R$ and $\lvert t\rvert \geq1$, we have \begin{equation}\label{StirlingEstimate}
		\lvert \Gamma(\sigma+it)\rvert \asymp e^{-\lvert t\rvert \frac\pi2}\lvert t\rvert ^{\sigma-1/2}.
	\end{equation} 
	
	We have Legendre's duplication formula
	\begin{equation}\label{Duplication formula}
		\Gamma(s)\Gamma(s+1/2)=2^{1-2s}\sqrt\pi\Gamma(2s)
	\end{equation} and Euler's reflection formula
	\begin{equation}\label{Reflection formula}
		\Gamma(1-s)\Gamma(s)=\frac{\pi}{\sin(\pi s)}.
	\end{equation}
	We will also use the formula \begin{equation}\label{RatioOfGammaFactors}
		\frac{\Gamma\bfrac{1-s}{2}}{\Gamma\bfrac s2}=\frac{2^s\sin (\pi s/2)\Gamma(1-s)}{\sqrt\pi}.
	\end{equation}
	Recall that $\Go(s)$ and $\Ge(s)$ denote the ratios of the gamma factors that appear in the functional equation for even or odd characters, so
	\begin{equation}\label{key}
		\Ge(s)=\frac{\Gamma\bfrac{1-s}{2}}{\Gamma\bfrac s2},
	\end{equation}
	and
	\begin{equation}\label{key}
		\Go(s)=\frac{\Gamma\bfrac{2-s}{2}}{\Gamma\bfrac {s+1}2}.
	\end{equation}
	
	\begin{lemma}\label{G odd + G even}
		We have \begin{equation}\label{}
			\Go(s)+\Ge(s)=\frac{2^{s+1/2}\Gamma(1-s)\cos\lz \frac{\pi s}{2}-\frac \pi 4\pz}{\sqrt \pi}.
		\end{equation}
	\end{lemma}
	\begin{proof}
		We have \begin{equation}\label{key}
			\Go(s)+\Ge(s)=\frac{\Gamma\lz\frac{2-s}{2}\pz\Gamma\lz\frac{s}{2}\pz+\Gamma\lz\frac{1-s}{2}\pz\Gamma\lz\frac{s+1}{2}\pz}{\Gamma\lz\frac{s+1}{2}\pz\Gamma\lz\frac{s}{2}\pz}.
		\end{equation} Using the reflection formula in the numerator and the duplication formula in the denominator, this equals
		\begin{equation}\label{key}
			\frac{2^{s-1}\sqrt\pi}{\Gamma(s)}\lz\frac{1}{\sin(\pi s/2)}+\frac{1}{\cos(\pi s/2)}\pz,
		\end{equation} and another application of the reflection formula in the denominator gives
		\begin{equation}\label{key}
			\frac{2^{s-1}\Gamma(1-s)\sin(\pi s)}{\sqrt\pi}\lz\frac{1}{\sin(\pi s/2)}+\frac{1}{\cos(\pi s/2)}\pz.
		\end{equation}
		The lemma follows after using the identity
		\begin{equation}\label{key}
			\sin(\pi s)\lz\frac{1}{\sin(\pi s/2)}+\frac{1}{\cos(\pi s/2)}\pz=2\sqrt 2\cos\lz\frac{\pi s}{2}-\frac{\pi}{4}\pz.
		\end{equation}
	\end{proof}	
	
	A key tool from multivariable complex analysis that we use is Bochner's Tube Theorem \cite{Boc}. For a set $U\subset\R^n$, we define $T(U)=U+i\R^n\subset \C^n.$ Then we have the following theorem:
	\begin{theorem}[Bochner's Tube Theorem]
		Let $U\subset\R^n$ be a connected open set and $f(z)$ be a function that is holomorphic on $T(U)$. Then $f(z)$ has a holomorphic continuation to the convex hull of $T(U)$.
	\end{theorem}
	A more general version of the theory of domains of holomorphy allows us to show that in the situation of Bochner's tube theorem, some of the properties of $f(z)$ also hold for its holomorphic continuation. An example that we include in Appendix C is Theorem \ref{Extending inequalities}, which we use to bound the continuation of some triple Dirichlet series in vertical strips. See also sections 4.3 and 4.4 of \cite{DGH}.
	
	\section{Overview of the proofs}\label{Section overview and heuristic}
	For simplicity, let us now assume our sums run over all positive integers. Our goal is to find an asymptotic formula for \begin{equation}\label{key}
		\sum_{n\leq X}\frac{L(1/2+\alpha,\chi_n)}{L(1/2+\beta,\chi_n)}.
	\end{equation} We insert a smooth weight into the sum and use Mellin inversion to rewrite it as an integral. Thus we obtain
	\begin{equation}\label{key}
		\sum_{n\geq1}\frac{L(1/2+\alpha,\chi_n)}{L(1/2+\beta,\chi_n)}f(n/X)=\frac{1}{2\pi i}\int_{(c)}A^*(s,1/2+\alpha,1/2+\beta)X^s\M f(s)ds,
	\end{equation} where $A^*(s,w,z)$ is the triple Dirichlet series
	\begin{equation}\label{eqn: MDS heuristic}
		A^*(s,w,z)=\sum_{n\geq1}\frac{L(w,\chi_n)}{L(z,\chi_n)n^s}=\sum_{m,n,k\geq1}\frac{\mu(k)\chi_n(k)\chi_n(m)}{n^sm^wk^z}.
	\end{equation} To be able to evaluate the integral, we need to investigate the analytic properties of $A^*(s,w,z)$. Assuming GRH and using \eqref{Lindelof} and \eqref{1/L under GRH}, the middle series of \eqref{eqn: MDS heuristic} series is absolutely convergent (up to a simple pole at $w=1$) in the region
	\begin{equation}\label{key}
		\{(s,w,z)\in\C^3:\re(s)>1,\ \re(s+w)>3/2,\ \re(z)>1/2\},
	\end{equation} where the second condition comes after using the functional equation for the L-function in the numerator if $\re(w)<1/2$ (neglecting for now that it only works for primitive characters). We remark that we get the same region without assuming GRH, as we only need the Lindelöf bound on average (see for example \cite{Blo} for details), which is provided by moments estimates.
	
	Comparing the integral with the prediction of the ratios conjecture, we may expect that $A^*(s,w,z)$ has a meromorphic continuation to $\re(s)>1/2$, with poles at $s=1$ and $s=1-\alpha=3/2-w.$
	
	We now present a heuristic of Bump, Friedberg and Hoffstein (see \cite{BFH}) suggesting that $A^*(s,w,z)$ satisfies some functional equations that can be used to obtain a meromorphic continuation, and that it has the predicted poles. In these heuristics, we assume that all characters are primitive, and that the quadratic reciprocity holds in the form $\chi_n(m)\approx \chi_m(n)$. We also don't write the gamma factors in the functional equations, so they have the form
	\begin{equation}\label{key}
		L(s,\chi_n)\approx n^{1/2-s}L(1-s,\chi_n).
	\end{equation}
	Using this functional equation in $w$, we obtain the first heuristic functional equation for $A^*(s,w,z)$:
	\begin{equation}\label{heuristic functional equation in w}
		A^*(s,w,z)\approx\sum_n\frac{L(w,\chi_n)}{L(z,\chi_n)n^s}\approx\sum_n\frac{L(1-w,\chi_n)}{L(z,\chi_n)n^{s+w-\frac12}}\approx A^*(s+w-1/2,1-w,z).
	\end{equation}
	On the other hand, we can expand the $L-$functions into Dirichlet series, first sum over the $n$ variable, and use the functional equation in $s$. Thus we obtain the second heuristic functional equation:
	\begin{equation}\label{heuristic functional equation in s}
		\begin{aligned}
			A^*(s,w,z)&\approx\sum_{m,n,k}\frac{\mu(k)\chi_n(k)\chi_n(m)}{n^sm^wk^z}\approx\sum_{m,k}\frac{\mu(k)L(s,\chi_{mk})}{m^wk^z}\\
			&\approx\sum_{m,k}\frac{\mu(k)L(1-s,\chi_{mk})}{m^{s+w-\frac12}k^{s+z-\frac12}}\approx A^*\lz1-s,s+w-\frac12,s+z-\frac12\pz.
		\end{aligned}
	\end{equation} We also see from this computation that there is a pole at $s=1$ coming from the terms with $mk=\square$, and \eqref{heuristic functional equation in w} then gives the pole at $s=3/2-w.$ We can also use \eqref{heuristic functional equation in w} to see a pole at $w=1$, which becomes a pole at $s=3/2-w$ after \eqref{heuristic functional equation in s}.
	
	An important aspect of our result is the admissible range of the parameters $\alpha$ and $\beta$, and the error term. This depends on the region of meromorphic continuation of $A^*(s,w,z),$ which we obtain from the functional equations and a careful application of Bochner's Tube Theorem.
	
	\medskip
	
	Let us now be more precise. For Theorem \ref{Theorem for fundamental discriminants}, we rewrite $R_D(X,\alpha,\beta;f)$ as an integral as
	\begin{equation}\label{R_D as integral}
		R_D(X,\alpha,\beta;f)=\frac1{2\pi i}\int_{(2)}A_D\lz s,\frac12+\alpha,\frac12+\beta\pz X^s\M f(s)ds,
	\end{equation} where $A_D(s,w,z)$ is the triple Dirichlet series
	\begin{equation}\label{key}
		\begin{aligned}
			A_D(s,w,z&)=\sumstar_{d\geq 1}\frac{L(w,\chi_d)}{L(z,\chi_d)d^s}=\sumstar_{d\geq1}\sum_{m,k\geq1}\frac{\mu(k)\chi_d(k)\chi_d(m)}{k^zm^wd^s}\\
			&=\sum_{m,k\geq1}\frac{\mu(k)L_D\lz s,\leg{\cdot}{mk}\pz}{m^wk^z}.
		\end{aligned}
	\end{equation}
	Since all characters are even and primitive, it is straightforward to get the functional equation in $w$. However, after exchanging summations, we obtain $L_D(s,\chi)$ instead of $L(s,\chi)$ in the heuristic in \eqref{heuristic functional equation in s}, so we don't have the functional equation in $s$. This explains the weaker results for this family. The details are written in Section \ref{Section series for fund discr}.
	
	For Theorem \ref{Theorem for all characters}, we similarly obtain
	\begin{equation}\label{Integral for all characters}
		\sum_{\substack{n\geq1,\\n\odd}}\frac{L_{(2)}(1/2+\alpha,\chi_n)}{L_{(2)}(1/2+\beta,\chi_n)}f\bfrac nX=\frac1{2\pi i}\int_{(2)}A\lz s,\frac12+\alpha,\frac12+\beta\pz X^s\M f(s)ds,
	\end{equation} where
	\begin{equation}\label{key}
		\begin{aligned}
			A(s,w,z)&=\sum_{\substack{n\geq1,\\n\odd}}\frac{L_{(2)}(w,\chi_n)}{L_{(2)}(z,\chi_n)n^s}=\sum_{\substack{k,m,n\geq1,\\k,m,n\odd}}\frac{\mu(k)\chi_n(k)\chi_n(m)}{k^zm^wn^s}\\
			&=\sum_{\substack{m,k\geq1,\\m,k\odd}}\frac{L\lz s,\leg{4mk}{\cdot}\pz}{m^wk^z}.
		\end{aligned}
	\end{equation} In this case, we can potentially obtain both functional equations, but we have to deal with the presence of non-primitive characters. A key tool here is the functional equation in Proposition \ref{Functional equation with Gauss sums}, which is valid for all Dirichlet characters. After applying it, we obtain a relation between $A$ and some different triple Dirichlet series, whose coefficients involve Gauss sums. We elaborate on this in Section \ref{Section series for all characters}.
	
	\section{The triple Dirichlet series for fundamental discriminants}\label{Section series for fund discr}
	As in \eqref{R_D as integral} we can write $R_D(X,\alpha,\beta;f)$ as
	\begin{equation}\label{key}
		R_D(X,\alpha,\beta;f)=\frac{1}{2\pi i}\int_{(c)}A_D(s,w,z)X^{s}\M f(s)ds,
	\end{equation} where 
	\begin{equation}\label{eqn: A_D definition}
		A_D(s,w,z)=\sumstar_{d\geq1}\frac{L(w,\chi_d)}{L(z,\chi_d)d^s},
	\end{equation} and $w=1/2+\alpha$, $z=1/2+\beta$.
	In this section, we investigate the analytic properties of the triple Dirichlet series $A_D(s,w,z)$ that allow us to shift the integral and compute the main terms coming from the two poles of $A_D(s,w,z)$.

	\subsection{Region of absolute convergence}
	Using \eqref{Lindelof}, \eqref{1/L under GRH} and the functional equation for $L(w,\chi_d)$ if $\re(w)<1/2$, up to a simple pole at $w=1$, the series \eqref{eqn: A_D definition} converges absolutely in the region
	\begin{equation}\label{key}
		R_0=\left\{(s,w,z):\ \re(s)>1,\ \re(s+w)>\frac32,\ \re(z)>\frac12\right\},
	\end{equation} and is polynomially bounded in vertical strips in this region, i.e., away from the possible poles, we have
	\begin{equation}\label{key}
		\lvert A_D(s,w,z)\rvert \ll_{\re(s),\re(w),\re(z)}((1+\lvert s\rvert )(1+\lvert w\rvert )(1+\lvert z\rvert ))^c
	\end{equation} for some constant $c$.
	By exchanging summations and using Lemma \ref{lemma - Dirichlet series of fund discr}, we have
	\begin{equation}\label{A_D(s,w,z) exchanged summation}
		\begin{aligned}
			A_D(s,w,z)&=\sum_{m,k\geq1}\sumstar_{d\geq1}\frac{\mu(k)\chi_d(k)\chi_d(m)}{d^sm^wk^z}=\sum_{m,k\geq 1}\frac{\mu(k)L_D\lz s,\leg{\cdot}{mk}\pz}{m^wk^z}\\
			&=\sum_{m,k\geq1}\lz\frac12+\frac{\leg{4}{mk}}{2\cdot 4^s}+\frac{\leg{8}{mk}}{8^s}\pz\frac{\mu(k)L\lz s,\leg{\cdot}{mk}\psi_1\pz}{m^wk^zL\lz 2s,\leg{\cdot}{mk}^2\psi_1\pz}\\
			&+\sum_{m,k\geq1}\lz\frac12-\frac{\leg{4}{mk}}{2\cdot 4^s}\pz\frac{\mu(k)L\lz s,\leg{\cdot}{mk}\psi_{-1}\pz}{m^wk^zL\lz2s,\leg{\cdot}{mk}^2\psi_1\pz}.
		\end{aligned}
	\end{equation}
	This expression converges absolutely in the region \begin{equation}\label{key}
		\begin{aligned}
			R_1=\Big\{(s,w,z):\ \re(s)&>\frac14,\ \re(w)>1,\ \re(z)>1,\\&\re(s+w)>\frac32,\ \re(s+z)>\frac32\Big\},
		\end{aligned}
	\end{equation} except there is a pole at $s=1$ coming from the terms in the first sum when $mk=\square$. Moreover, $(s-1)A_D(s,w,z)$ is polynomially bounded in vertical strips in $R_1$.
	Bochner's tube theorem allows us to meromorphically continue $A_D(s,w,z)$ to the convex hull of $R_0$ and $R_1$, which is
	\begin{equation}\label{key}
		R_2=\left\{(s,w,z):\ \re(s)>\frac14,\ \re(z)>\frac12,\ \re(s+w)>\frac32,\ \re(s+z)>\frac32\right\},
	\end{equation} and by Proposition \ref{Extending inequalities}, $(s-1)A_D(s,w,z)$ is polynomially bounded in vertical strips in $R_2$.

	\subsection{Residue at $s=1$}
	We see from expression \eqref{A_D(s,w,z) exchanged summation} that $A_D(s,w,z)$ has a pole at $s=1$ coming from the terms with $mk=\square$. In this case, we have
	\begin{equation}\label{key}
		\frac{L\lz s,\leg{\cdot}{mk}\psi_1\pz}{L\lz 2s,\leg{\cdot}{mk}^2\psi_1\pz}%=\frac{\zeta(s)}{\zeta(2s)}\prod_{p\mid 2mk}\frac{1-p^{-s}}{1-p^{-2s}}
		=\frac{\zeta(s)}{\zeta(2s)}\prod_{p\mid 2mk}\frac{p^s}{p^s+1}.
	\end{equation}
	If $mk$ is an odd square, then $mk\equiv1\mod 8$, so
	\begin{equation}\label{key}
		\lz\frac12+\frac{\leg{4}{mk}}{8}+\frac{\leg{8}{mk}}{8}\pz\prod_{p\mid 2mk}\frac{p}{p+1}=\frac34\cdot\frac23\prod_{p\mid mk}\frac{p}{p+1}=\frac12\prod_{p\mid mk}\frac{p}{p+1},
	\end{equation} and the same holds if $mk$ is an even square. Therefore, we have
	\begin{equation}\label{key}
		\res_{s=1}A_D(s,w,z)=\frac{1}{2\zeta(2)}\sum_{mk=\square}\frac{\mu(k)}{m^wk^z}\prod_{p\mid mk}\frac{p}{p+1}.
	\end{equation} This is the same expression as in the heuristic computation (2.21) in \cite{CoSn}, and can be written as
	\begin{equation}\label{key}
		\res_{s=1}A_D(s,w,z)=\frac{\zeta(2w)}{2\zeta(2)\zeta(z+w)} P_D(z,w),
	\end{equation} where
	\begin{equation}\label{key}
		\begin{aligned}
			P_D(z,w)&=\prod_{p}\lz1-\frac1{p^{z+w}}\pz^{-1}\lz1-\frac{1}{(p+1)p^{2w}}-\frac{p}{(p+1)p^{z+w}}\pz\\
			&=\prod_{p}\lz1+\frac{1-p^{z-w}}{\lz p^{z+w}-1\pz(p+1)}\pz.
		\end{aligned}
	\end{equation}

	\subsection{Functional equation and meromorphic continuation}
	We now use the functional equation of $L(w,\chi_d)$ to obtain a functional equation for $A_D(s,w,z)$, which gives us a meromorphic continuation beyond the region $R_2$. All of the characters $\chi_d$ are even and primitive of conductor $d$, so \eqref{Functional equation for L-functions of primitive characters} gives
	\begin{equation}\label{functional equationf for A_D(s,w,z)}
		\begin{aligned}
			A_D(s,w,z)&=\sumstar_{d\geq1}\frac{L(w,\chi_d)}{L(z,\chi_d)d^s}=\pi^{w-1/2}\Ge(w)\sumstar_{d\geq1}\frac{L(1-w,\chi_d)}{L(z,\chi_d)d^{s+w-1/2}}\\
			&=\pi^{w-1/2}\Ge(w)A_D(s+w-1/2,1-w,z).
		\end{aligned}
	\end{equation}
	This functional equation provides a meromorphic continuation to the region
	\begin{equation}\label{key}
		R_3=\big\{(s,w,z):\ \re(s+w)>\frac34,\ \re(z)>\frac12,\ \re(s)>1,\ \re(s+w+z)>2\big\},
	\end{equation}  and gives rise to a new pole of $A_D(s,w,z)$ at $s=3/2-w$ with residue
	\begin{equation}\label{key}
		\res_{s=3/2-w}A_D(s,w,z)=\pi^{w-1/2}\Ge(w)\frac{\zeta(2-2w)}{2\zeta(2)\zeta(1+z-w)}P_D(z,1-w).
	\end{equation}
	Bochner's tube theorem allows us to meromorphically continue $A_D(s,w,z)$ to the convex hull of $R_2$ and $R_3$, which is the region
	\begin{equation}\label{Final region for A_D(s,w,z)}
		\begin{aligned}
			R_4=\Bigg\{(s,w,z)&:\ \re(s)>\frac14,\ \re(z)>\frac12,\ \re(s+w)>\frac34,\ \re(2s+w)>\frac74,\\
			&\re(s+z)>\frac32,\ \re(2s+w+z)>3,\ \re(s+w+z)>2 
			\Bigg\}.
		\end{aligned}
	\end{equation}
	Moreover, $(s-1)(w-1)(s+w-3/2)A_D(s,w,z)$ is polynomially bounded in vertical strips in the region $R_3$, and by Proposition \ref{Extending inequalities} also in the region $R_4$.
	
	\section{Proof of Theorem \ref{Theorem for fundamental discriminants}}\label{Section Proof of Theorem 1}
	In this section, we prove Theorem \ref{Theorem for fundamental discriminants}. Assume that $-1/2<\re(\alpha)<1/2$, $0<\re(\beta)<1/2,$ and $\re(\beta)>\lvert \re(\alpha)\rvert $.
	We have
	\begin{equation}\label{key}
		R_D(X,\alpha,\beta;f)=\frac{1}{2\pi i}\int_{(2)}A_D(s,w,z) X^s\M f(s)ds,
	\end{equation} where $w=1/2+\alpha$, $z=1/2+\beta$.
	Let \begin{equation}\label{key}
		\begin{aligned}
			&M(\alpha,\beta)=1-\frac{\re(\alpha)}{2}-\frac{\re(\beta)}{2}.
		\end{aligned}
	\end{equation} By \eqref{Final region for A_D(s,w,z)}, we can shift the integral to
	$\re(s)=M(\alpha,\beta)+\epsilon,$
	and our assumptions about $\re(\alpha)$, $\re(\beta)$ ensure that we cross the poles at $s=1$, $s=1-\alpha$, and that the Euler products $P_D\lz\frac12+\beta,\frac12\pm\alpha\pz$ converge absolutely.
	\begin{equation}\label{key}
		\begin{aligned}
			R_D(X,\alpha,\beta;f)&=X\M f(1)\res_{s=1}A_D(s,1/2+\alpha,1/2+\beta)\\
			&+X^{1-\alpha}\M f(1-\alpha)\res_{s=1-\alpha}A_D(s,1/2+\alpha,1/2+\beta)\\
			&+\frac1{2\pi i}\int_{(M(\alpha,\beta)+\epsilon)}A_D(s,1/2+\alpha,1/2+\beta)\M f(s)X^sds.
		\end{aligned}
	\end{equation}
	Since $A_D(s,w,z)$ is polynomially bounded in vertical strips, we can bound the last integral and obtain
	\begin{equation}
		\begin{aligned}
			R_D(X,\alpha,\beta;f)&=\frac{X\M f(1)\zeta(1+2\alpha)}{2\zeta(2)\zeta(1+\alpha+\beta)}P_D(1/2+\beta,1/2+\alpha)\\
			&+\frac{X^{1-\alpha}\M f(1-\alpha)\zeta(1-2\alpha)\pi^\alpha\Ge(1/2+\alpha)}{2\zeta(2)\zeta(1-\alpha+\beta)}P_D(1/2+\beta,1/2-\alpha)\\
			&+O_{\alpha,\beta}(X^{M(\alpha,\beta)+\epsilon}).
		\end{aligned}
	\end{equation}
	
	\section{The triple Dirichlet series for all characters}\label{Section series for all characters}
	
	We now proceed with proving the results for the family with non-primitive characters. In view of \eqref{Integral for all characters}, we begin by studying the properties of the triple Dirichlet series $A(s,w,z)$.
	
	\subsection{Region of absolute convergence}
	Writing $n=n_0n_1^2$ with $n_0$ squarefree, we have
	\begin{equation}\label{key}
		\begin{aligned}
			A(s,w,z)&=\sum_{\substack{n\geq1,\\n\odd}}\frac{L_{(2)}(w,\chi_n)}{L_{(2)}(z,\chi_n)n^s}=\sum_{n_0,n_1\odd}\frac{\mu^2(n_0)L_{(2)}(w,\chi_{n_0n_1^2})}{L_{(2)}(z,\chi_{n_0n_1^2})n_0^sn_1^{2s}}\\
			&=\sum_{n_1\odd}\frac{1}{n_1^{2s}}\sum_{n_0\odd}\frac{\mu^2(n_0)L_{(2)}(s,\chi_{n_0})}{L_{(2)}(z,\chi_{n_0n_1^2})n_0^s}\prod_{p|n_1}\lz1-\frac{\chi_{n_0}(p)}{p^w}\pz.
		\end{aligned}
	\end{equation} 
	Using
	\begin{equation}\label{key}
		\prod_{p|n_1}\lz1-\frac{\chi_{n_0}(p)}{p^w}\pz\ll n_1^{\epsilon+\max\{0,-\re(w)\}},
	\end{equation} the above sum can be bounded by
	\begin{equation}\label{key}
		\sum_{n_1\odd}\frac{n_1^{\epsilon+\max\{0,-\re(w)\}}}{n_1^{2\re(s)}}\sum_{n_0\odd}\frac{\mu^2(n_0)|L(w,\chi_{n_0})|}{|L(z,\chi_{n_0n_1^2})|n_0^{\re(s)}}
	\end{equation}
	By \eqref{Lindelof} and $\eqref{1/L under GRH}$, the double sum is absolutely convergent in the region
	\begin{equation}\label{eqn: A initial partial region 1}
		\{(s,w,z):\hbox{$\re(s)>1,\ \re(w)\geq 1/2,\ \re(z)>1/2$}\}.
	\end{equation} If $\re(w)<1/2$, we may use the functional equation for $L(w,\chi_{n_0})$, obtaining the region of absolute convergence 
	\begin{equation}\label{eqn: A initial partial region 2}
		\{(s,w,z):\hbox{$\re(w)<1/2,\ \re(s+w)>3/2,\ \re(z)>1/2$}\}
	\end{equation} (note that in the above region, we have $\re(s)>1/2$, $\re(2s+w)>1$).
	
	Combining \eqref{eqn: A initial partial region 1} and \eqref{eqn: A initial partial region 2}, we see that $A(s,w,z)$ can be initially defined in the region
	\begin{equation}\label{key}
		S_0=\{(s,w,z):\hbox{$\re(s)>1,\ \re(s+w)>3/2,\ \re(z)>1/2$}\}.
	\end{equation}
	
	We also have
	\begin{equation}\label{Sum A(s,w,z) over n}
		A(s,w,z)=\sum_{\substack{m,n,k\geq1,\\m,n,k\odd}}\frac{\chi_n(m)\chi_n(k)\mu(k)}{n^sm^wk^z}=\sum_{\substack{m,k\geq1,\\m,k\odd}}\frac{\mu(k)L\lz s,\leg{4mk}{\cdot}\pz}{m^wk^z},
	\end{equation}
	which is similarly defined in the region 
	\begin{equation}\label{key}
		S_1=\{(s,w,z):\hbox{$\re(w)>1,\ \re(z)>1,\ \re(s+w)>\frac32,\ \re(s+z)>\frac32$}\},
	\end{equation} except the pole at $s=1$ coming from the summands with $mk=\square$.
	The convex hull of $S_0$ and $S_1$ is 
	\begin{equation}\label{Region of convergence of A(s,w,z)}
		S_2=\{(s,w,z):\re(z)>1/2,\ \re(s+w)>3/2,\ \re(s+z)>3/2\}.
	\end{equation}
	\subsection{Pole and residue at $s=1$}
	
	We see from \eqref{Sum A(s,w,z) over n} that $A(s,w,z)$ has a pole at $s=1$ coming from the summands where $mk=\square.$ In this case, we have
	\begin{equation}
		L\lz s,\leg{4mk}{\cdot}\pz=\zeta(s)\prod_{p\mid 4mk}\lz1-\frac1{p^s}\pz,
	\end{equation} so denoting by $a(n)$ the multiplicative function with $a(p^k)=1-\frac 1p$, we have
	\begin{equation}\label{key}
		\res_{s=1}A(s,w,z)=\sum_{\substack{mk=\square,\\m,k\odd}}\frac{\mu(k)a(4mk)}{m^wk^z}=\frac12\sum_{\substack{mk=\square,\\m,k\odd}}\frac{\mu(k)a(mk)}{m^wk^z}.
	\end{equation}
	We can write the last sum as an Euler product, slightly abusing notation by writing $p^k$ for the prime factors of $k$, and similarly for $m$. We thus obtain 
	\begin{equation}
		\begin{aligned}
			&\frac12\prod_{p>2}\sum_{\substack{m,k\geq0,\\m+k\even}}\frac{\mu(p^k)a(p^{m+k})}{p^{mw+kz}}\\
			&=\frac12\prod_{p>2}\lz\sum_{\substack{m\geq0,\\m\even}}\frac{a(p^m)}{p^{mw}}-\sum_{\substack{m\geq0,\\m\odd}}\frac{a(p^{m+1})}{p^{z+mw}}\pz\\
			&=\frac12\prod_{p>2}\lz1+\lz1-\frac1p\pz\cdot\frac{p^{-2w}}{1-p^{-2w}}-\lz1-\frac1p\pz\frac{p^{-z-w}}{1-p^{-2w}}\pz\\
			%&=\frac{\zeta_{(2)}(2w)}{2}\prod_{p>2}\lz1+\lz1-\frac1p\pz\cdot\frac{p^{-2w}}{1-p^{-2w}}-\lz1-\frac1p\pz\frac{p^{-z-w}}{1-p^{-2w}}\pz\lz1-p^{-2w}\pz=\\
			%&=\frac{\zeta_{(2)}(2w)}{2}\prod_{p>2}\lz1-p^{-2w}+\lz1-\frac1p\pz p^{-2w}-\lz1-\frac1p\pz p^{-z-w}\pz=\\
			%&=\frac{\zeta_{(2)}(2w)}{2}\prod_{p>2}\lz1-\frac1{p^{2w+1}}-\frac{p-1}{p^{1+z+w}}\pz=\\
			%&=\frac{\zeta_{(2)}(2w)}{2\zeta_{(2)}(z+w)}\prod_{p>2}\lz1-\frac1{p^{2w+1}}-\frac{p-1}{p^{1+z+w}}\pz\lz1-\frac1{p^{z+w}}\pz^{-1}=\\
			%&=\frac{\zeta_{(2)}(2w)}{2\zeta_{(2)}(z+w)}\prod_{p>2}\lz1-\frac1{p^{2w+1}\lz1-p^{-z-w}\pz}+\frac{\frac1{p^{z+w}}-\frac{p-1}{p^{z+w+1}}}{1-p^{-z-w}}\pz=\\
			&=\frac{\zeta_{(2)}(2w)}{2\zeta_{(2)}(z+w)}\prod_{p>2}\lz1-\frac{1}{p^{1+w-z}(p^{z+w}-1)}+\frac{1}{p(p^{z+w}-1)}\pz\\
			&=\frac{\zeta_{(2)}(2w)}{2\zeta_{(2)}(z+w)}\prod_{p>2}\lz1+\frac{p^{w-z}-1}{p^{1+w-z}(p^{z+w}-1)}\pz.
		\end{aligned}
	\end{equation}
	Setting $w=1/2+\alpha$, $z=1/2+\beta$ gives
	\begin{equation}\label{Residue at s=1}
		\begin{aligned}
			\res_{s=1}&A(s,1/2+\alpha,1/2+\beta)\\
			&=\frac{\zeta_{(2)}(1+2\alpha)}{2\zeta_{(2)}(1+\alpha+\beta)}\prod_{p>2}\lz1+\frac{p^{\alpha-\beta}-1}{p^{1+\alpha-\beta}(p^{1+\alpha+\beta}-1)}\pz.
		\end{aligned}
	\end{equation}
	
	\subsection{Functional equation in $s$}
	To get a functional equation for $A(s,w,z)$, we use expression \eqref{Sum A(s,w,z) over n} together with the functional equation from Proposition \ref{Functional equation with Gauss sums}. Since $\leg{4mk}{\cdot}$ is an even Dirichlet character modulo $4mk$  for any $m,k\geq1$, we obtain 
	\begin{equation}\label{Functional equation in s}
		\begin{aligned}
			A(s,w,z)&=\sum_{\substack{m,k\geq1,\\m,k\odd}}\frac{\mu(k)L\lz s,\leg{4mk}{\cdot}\pz}{m^wk^z}\\
			&=\frac{\pi^{s-1/2}\Gamma\bfrac{1-s}{2}}{4^s\Gamma\bfrac{s}{2}}\sum_{\substack{m,k\geq1,\\m,k\odd}}\frac{\mu(k)K\lz1-s,\leg{4mk}{\cdot}\pz}{m^{s+w}k^{s+z}}\\
			&= \frac{\pi^{s-1/2}\Gamma\bfrac{1-s}{2}}{4^s\Gamma\bfrac{s}{2}} C(1-s,s+w,s+z),
		\end{aligned}
	\end{equation}
	where $C(s,w,z)$ is the triple Dirichlet series
	\begin{equation}\label{key}
		C(s,w,z)=\sum_{\substack{m,k\geq1,\\m,k\odd}}\frac{\mu(k)K\lz s,\leg{4mk}{\cdot}\pz}{m^wk^z}=\sum_{\substack{m,k,q\geq 1,\\m,k\odd}}\frac{\mu(k)\tau\lz\leg{4mk}{\cdot},q\pz}{q^sm^wk^z}
	\end{equation}
	
	\subsection{Region of convergence of $C(s,w,z)$}
	
	By \eqref{Region of convergence of A(s,w,z)} and the functional equation \eqref{Functional equation in s}, $C(s,w,z)$ is initially defined in the region \begin{equation}\label{key}
		P=\{(s,w,z):\ \re(s+z)>3/2,\ \re(w)>3/2,\ \re(z)>3/2\}.
	\end{equation}
	To extend this region, we exchange the summations in $C(s,w,z)$ and get \begin{equation}\label{Exchanging sum in C}
		C(s,w,z)=\sum_{\substack{m,k,q\geq 1,\\m,k\odd}}\frac{\mu(k)\tau\lz\leg{4mk}{\cdot},q\pz}{q^sm^wk^z}=\sum_{q=1}^\infty\frac{1}{q^s}\sum_{\substack{\l\geq1,\\\l\odd}}\frac{\tau\lz\leg{4\l}{\cdot},q\pz}{\l^w}\sum_{k\mid \l}\frac{\mu(k)}{k^{z-w}},
	\end{equation} where we also substituted $mk=\l$. Our goal now is to study the properties of the inner Dirichlet series, that is, the sum over $\l.$
	
	We denote by $a_t(\l)$ the multiplicative function with $a_t(p^k)=1-\frac1{p^t}$ and use Lemma \ref{Lemma changing Gauss sums}, so we can rewrite \eqref{Exchanging sum in C} as 
	\begin{equation}\label{C(s,w,z) simplified}
		\begin{aligned}
			C&(s,w,z)=\sum_{\substack{\l\equiv1\mod 4,\\ q\equiv 2\mod 4}}\frac{-2\tau\lz\leg{\cdot}{\l},q\pz a_{z-w}(\l)}{\l^wq^s}+\sum_{\substack{\l\equiv1\mod 4,\\ q\equiv 0\mod 4}}\frac{2\tau\lz\leg{\cdot}{\l},q\pz a_{z-w}(\l)}{\l^wq^s}\\
			&+\sum_{\substack{\l\equiv3\mod 4,\\ q\equiv 1\mod 4}}\frac{-2i\tau\lz\leg{\cdot}{\l},q\pz a_{z-w}(\l)}{\l^wq^s}
			+\sum_{\substack{\l\equiv3\mod 4,\\ q\equiv 3\mod 4}}\frac{2i\tau\lz\leg{\cdot}{\l},q\pz a_{z-w}(\l)}{\l^wq^s}\\
			&=\frac{-2}{2^s}\sum_{\substack{\l\equiv1\mod 4,\\ q\odd}}\frac{\leg{2}{\l}\tau\lz\leg{\cdot}{\l},q\pz a_{z-w}(\l)}{\l^wq^s}+\frac{2}{4^s}\sum_{\substack{\l\equiv1\mod 4,\\ q\geq1}}\frac{\tau\lz\leg{\cdot}{\l},q\pz a_{z-w}(\l)}{\l^wq^s}\\
			&+\sum_{\substack{\l\equiv3\mod 4,\\ q\equiv 1\mod 4}}\frac{-2i\tau\lz\leg{\cdot}{\l},q\pz a_{z-w}(\l)}{\l^wq^s}
			+\sum_{\substack{\l\equiv3\mod 4,\\ q\equiv 3\mod 4}}\frac{2i\tau\lz\leg{\cdot}{\l},q\pz a_{z-w}(\l)}{\l^wq^s}
		\end{aligned}
	\end{equation}
	All of the terms in \eqref{C(s,w,z) simplified} can be written as combinations of terms of the form \begin{equation}\label{key}
		C(s,w,z;\psi,\psi'):=\sum_{\l,q\geq 1}\frac{G\lz\leg{\cdot}{\l},q\pz\psi(\l)\psi'(q)}{\l^wq^s}\sum_{k\mid \l}\frac{\mu(k)}{k^{z-w}},
	\end{equation} where $\psi,\psi'$ are characters modulo $b\mid 8$, and $\psi(2)=0$. In particular, we have
	\begin{equation}\label{C(s,w,z) as twisted C(s,w,z)}
		\begin{aligned}
			C(s,w,z)&=
			-2^{-s}\big( C(s,w,z;\psi_2,\psi_1)+C(s,w,z;\psi_{-2},\psi_1)\big)\\
			&+4^{-s}\big( C(s,w,z;\psi_1,\psi_0)+C(s,w,z;\psi_{-1},\psi_0)\big)\\
			&+C(s,w,z;\psi_1,\psi_{-1})-C(s,w,z;\psi_{-1},\psi_{-1}).
		\end{aligned}
	\end{equation}
	
	In $C(s,w,z;\psi,\psi')$, all coefficients are multiplicative in $\l$.
	We write
	\begin{equation}\label{key}
		C(s,w,z;\psi,\psi')=\sum_{q=1}^\infty\frac{\psi'(q)}{q^s}\cdot D(w,z-w,q;\psi),
	\end{equation}
	where 
	\begin{equation}\label{key}
		D(w,t,q;\psi)=\sum_{\l=1}^\infty\frac{G\lz\leg{\cdot}{\l},q\pz\psi(\l)a_{t}(\l)}{\l^w}.
	\end{equation} We have the following lemma:

	\begin{lemma}\label{Estimate For D(w,t)}
		\begin{enumerate}
			\item $D(w,t,q;\psi)$ has a meromorphic continuation to the region \begin{equation}\label{key}
				\{(w,t):\ \re(w)>1,\ \re(w+t)>1\}.
			\end{equation}
			The only pole in this region is at $w=3/2$, which occurs if $q=\square,$ $\psi$ is a principal character and $t\neq 0.$
			\item For $\re(w)>1+\epsilon$ and $\re(t+w)>1+\epsilon$, away from the possible poles, we have \begin{equation}\label{key}
				\lvert D(w,t,q,\psi)\rvert \ll\lvert w(t+w)\rvert ^{\epsilon} q^{\max\{\epsilon,\epsilon-\re(t)\}}.
			\end{equation}		
		\end{enumerate}
	\end{lemma}
	
	\begin{proof}
		We can write $D(w,t,q;\psi)$ as an Euler product
		\begin{equation}\label{key}
			\begin{aligned}
				D(w,t,q;\psi)&=\prod_p\lz\sum_{k=0}^\infty\frac{G\lz\leg{\cdot}{p^k},q\pz\psi(p^k)a_t(p^k)}{p^{kw}}\pz\\
				&=P_{p\nmid q}(w,t,q;\psi)P_{p\mid q}(w,t,q;\psi),
			\end{aligned}
		\end{equation}
		where $P_{p\nmid q}$ is the product over odd primes not dividing $q$, and $P_{p\mid q}$ is the rest. Since $\psi(2)=0$, we can only consider odd primes. Also note that $P_{p\mid q}$ is a finite product, and each factor has only finitely many terms by \eqref{Sound's Gauss sums - exact formula}.
		
		For an odd $p\nmid q,$ we have
		\begin{equation}\label{key}
			G\lz\leg{\cdot}{p^k},q\pz=\begin{cases}
				1,&\hbox{if $k=0$,}\\
				\leg{q}{p}\sqrt p,&\hbox{if $k=1$,}\\
				0,&\hbox{if $k\geq2$,}
			\end{cases}
		\end{equation}
		so for $P_{p\nmid q}$, we have
		\begin{equation}\label{key}
			\begin{aligned}
				P_{p\nmid q}(w,t,q;\psi)&=\prod_{p\nmid q}\lz1+\frac{\leg{4q}{p}\psi(p)}{p^{w-1/2}}\pz\prod_{p\nmid q}\frac{1+\frac{\leg{4q}{p}\psi(p)(1-p^{-t})}{p^{w-1/2}}}{1+\frac{\leg {4q}p\psi(p)}{p^w-1/2}}\\
				%&=\prod_{p}\lz1-\frac{\leg{4q}{p}^2\psi(p)^2}{p^{2w-1}}\pz\lz1-\frac{\leg{4q}{p}\psi(p)}{p^{w-1/2}}\pz^{-1}\prod_{p}\lz1-\frac{\leg{4q}{p}\psi(p)}{p^{t+w-1/2}}\cdot\frac{1}{1+\frac{\leg{q}{p}\psi(p)\sqrt p}{p^w}}\pz=\\
				&=\frac{L\lz w-1/2,\leg{4q}{\cdot}\psi\pz}{\zeta_{(4q)}(2w-1)}E(w,t,q;\psi),
			\end{aligned}
		\end{equation} where \begin{equation}\label{key}
			\begin{aligned}
				&E(w,t,q;\psi)=\prod_{p}\lz1-\frac{\leg{4q}{p}\psi(p)}{p^{t+w-1/2}}\cdot\frac{1}{1+\frac{\leg{4q}{p}\psi(p)\sqrt p}{p^w}}\pz\\
				&=\frac{1}{L\lz t+w-1/2,\leg{4q}{\cdot}\psi\pz}\prod_p\lz\frac{1-\frac{\leg{4q}{p}\psi(p)}{p^{t+w-1/2}}\cdot\frac{1}{1+\leg{q}{p}\psi(p)p^{1/2-w}}}{1-\frac{\leg{4q}{p}\psi(p)}{p^{t+w-\frac12}}}\pz\\
				%&=\frac{1}{L\lz t+w-\frac12,\leg{4q}{\cdot}\psi\pz}\prod_p\lz1+\frac{\frac{\leg{4q}{p}\psi(p)}{p^{t+w-1/2}}\cdot\lz1-\frac1{1+\leg{4q}{p}\psi(p)p^{1/2-w}}\pz}{1-\frac{\leg{4q}{p}\psi(p)}{p^{t+w-1/2}}}\pz=\\
				%&=\frac1{L\lz t+w-\frac12,\leg{4q}{\cdot}\psi\pz}\prod_p\lz1+\frac{\leg{4q}{p}\psi(p)\frac{\leg {4q}p\psi(p)p^{1/2-w}}{1+\leg qp\psi(p)p^{1/2-w}}}{p^{t+w-1/2}-\leg{4q}{p}\psi(p)}\pz=\\
				&=\frac{1}{L\lz t+w-\frac12,\leg{4q}{\cdot}\psi\pz}\prod_{p}\lz1+\frac{\leg {4q}p^2}{\lz p^{t+w-\frac12}-\leg{4q}{p}\psi(p)\pz\lz p^{w-\frac12}+\leg{4q}{p}\psi(p)\pz}\pz.
			\end{aligned}
		\end{equation}
		For $\re(t+w)>1+\epsilon,\ \re(w)>1+\epsilon,$ the last Euler product is absolutely convergent and $\ll 1$. This finishes the proof of part (i) of the Lemma.
		
		To prove part (ii), we have to estimate the size of the remaining factors.

		We have \begin{equation}\label{key}
			\zeta_{(4q)}(2w-1)^{-1}=\zeta(2w-1)^{-1}\prod_{p\mid 4q}\lz1-\frac{1}{p^{2w-1}}\pz^{-1},
		\end{equation} and for $\re(w)>1/2+\epsilon$, the product can be bounded by
		\begin{equation}\label{bound for zeta function}
			\prod_{p\mid 4q}\lz1-\frac{1}{p^{2w-1}}\pz^{-1}=\prod_{p\mid 4q}\lz1+\frac{1}{p^{2w-1}-1}\pz\ll c^{\omega(4q)}\ll q^\epsilon,
		\end{equation} where $c$ is a suitable constant (depending on $\epsilon$), and the last bound follows from the elementary estimate $\omega(n)\ll\frac{\log n}{\log\log n}$ (here $\omega(n)$ denotes the number of prime factors of $n$).
		
		It remains to bound $P_{p\mid q}(w,t,q;\psi)$.
		By \eqref{Sound's Gauss sums - exact formula}, we can write it as
		\begin{equation}\label{Product over primes dividing q}
			\begin{aligned}
				&\prod_{p^a\mid \mid q}\lz1+\sum_{k=1}^{\lfloor \frac a 2\rfloor}\frac{\phi(p^{2k})\psi(p^{2k})}{p^{2kw}}(1-p^{-t})+\frac{G\lz\leg{\cdot}{p^{a+1}},q\pz\psi(p^{a+1})}{p^{(a+1)w}}(1-p^{-t})\pz\\
				&\ll\prod_{p\mid q}\lz1+(1-p^{-t})\sum_{k=1}^\infty p^{2k(1-w)}\pz,
			\end{aligned}
		\end{equation} and the geometric series is $\ll 1$ for $\re(w)>1+\epsilon$. We have 
		\begin{equation}\label{key}
			1-p^{-t}\ll\begin{cases}
				1,&\hbox{if $\re(t)\geq0$,}\\
				p^{-\re(t)},&\hbox{if $\re(t)<0$,}
			\end{cases}
		\end{equation} so the last expression in $\eqref{Product over primes dividing q}$ is \begin{equation}\label{key}
			\ll\prod_{p\mid q}\lz1+c_1\lz1+p^{-\re(t)}\pz\pz\ll\begin{cases}
				q^{\epsilon},&\hbox{if $\re(t)\geq0$,}\\
				q^{-\re(t)+\epsilon},&\hbox{if $\re(t)<0$,}
			\end{cases}
		\end{equation}
		where the bounds were obtained similarly as in \eqref{bound for zeta function}.
	\end{proof}
	Using this Lemma, we can extend each of $C(s,w,z;\psi,\psi')$, and hence also $C(s,w,z)$ to the region
	\begin{equation}\label{key}
		\{(s,w,z):\ \re(w)>1,\ \re(z)>1,\ \re(s)+\min\{0,\re(z-w)\}>1\}.
	\end{equation}
	Using the functional equation \eqref{Functional equation in s} enables us to extend $A(s,w,z)$ to 
	\begin{equation}\label{key}
		S_3=\{(s,w,z):\ \re(s+w)>1,\ \re(s+z)>1,\ \re(1-s)+\min\{0,\re(z-w)\}>1\}.
	\end{equation}
	The convex hull of $S_2$ and $S_3$ contains
	\begin{equation}\label{Final region of definition for A(s,w,z)}
		\begin{aligned}
			S_4=\Bigg\{(s,w,z):\ &\re(s+2w)>2,\ \re(s+2z)>2,\\
			&\re(s+z)>1,\ \re(s+w)>1,\ \re(z)>1/2\Bigg\}.
		\end{aligned}
	\end{equation}

	\subsection{Bounding $A(s,w,z)$ in vertical strips}\label{Section bound in vertical strips}
	
	We now give bounds for $\lvert A(s,w,z)\rvert $ in vertical strips, which are necessary to bound the error term coming from the shifted integral. To get the desired result, we need a bound of the form $\lvert wz\rvert ^{\epsilon}\lvert s\rvert ^K$ for some constant $K$, so we can be a little wasteful in the exponent of $s$.
	
	For the earlier defined regions $S_j$, we define \begin{equation}\label{Definition of S tilde}
		\tilde S_j=S_{j,\epsilon}\cap\{(s,w,z):\re(s)>-5/2,\ \re(w)>1/2-\epsilon\},
	\end{equation}
	where $S_{j,\epsilon}=S_j+\epsilon \vec{v}$, with  $\vec{v}=(1,1,1)$. Let also 
	\begin{equation}\label{key}
		p(s,w)=(s-1)(w-1)(s+w-3/2),
	\end{equation} so that $A(s,w,z)p(s,w)$ is an analytic function in the considered regions. We also denote $\tilde p(s,w)=1+\lvert p(s,w)\rvert $.
	
	We first give bounds in the restricted regions $\tilde S_0$, $\tilde S_1$. We have
	\begin{equation}\label{key}
		\lvert p(s,w)A(s,w,z)\rvert =\lab p(s,w)\sum_{\substack{n\geq1,\\n\odd}}\frac{L_{(2)}(w,\chi_n)}{L_{(2)}(z,\chi_n)n^s}\rab\ll\tilde p(s,w)\lvert wz\rvert ^{\epsilon},
	\end{equation}
	valid in $\tilde S_0$. Exchanging summations and using \eqref{Lindelof}, we also get the bound
	\begin{equation}\label{key}
		\begin{aligned}
			\lvert p(s,w)A(s,w,z)\rvert &=\lab p(s,w) \sum_{\substack{m,k\geq1,\\m,k\odd}}\frac{\mu(k)L\lz s,\leg{4mk}{\cdot}\pz}{m^wk^z}\rab\\
			&\ll\tilde p(s,w)(1+\lvert s\rvert )^{\max\{\epsilon,\frac12-\re(s)+\epsilon\}}.
		\end{aligned}
	\end{equation} which holds in  $\tilde S_1$. Using Proposition \ref{Extending inequalities}, we get the bound
	\begin{equation}\label{key}
		\lvert p(s,w)A(s,w,z)\rvert \ll \tilde p(s,w) \lvert wz\rvert ^{\epsilon}(1+\lvert s\rvert )^{3+\epsilon}
	\end{equation} in the convex hull of $\tilde S_0,\tilde S_1$, which is $\tilde S_2$.
	
	By \eqref{C(s,w,z) as twisted C(s,w,z)} and Lemma \ref{Estimate For D(w,t)}, we have 
	\begin{equation}\label{key}
		\lvert (w-3/2)C(s,w,z)\rvert \ll (1+\lvert w-3/2\rvert )\lvert wz\rvert ^{\epsilon}
	\end{equation} in the region \begin{equation}\label{key}
		\{(s,w,z):\re(w)>1+\epsilon,\ \re(z)>1+\epsilon,\ \re(s)+\min\{0,\re(z-w)>1+\epsilon\},
	\end{equation} so the functional equation \eqref{Functional equation in s} gives the bound
	\begin{equation}\label{key}
		\lvert p(s,w)A(s,w,z)\rvert \ll \tilde p(s,w) \lvert wz\rvert ^{\epsilon}(1+\lvert s\rvert )^{3+\epsilon}
	\end{equation} in the region $\tilde S_3.$ Using Proposition \ref{Extending inequalities} once more gives us the final bound
	\begin{equation}
		\lvert p(s,w)A(s,w,z)\rvert \ll \tilde p(s,w)\lvert wz\rvert ^{\epsilon}(1+\lvert s\rvert )^{3+\epsilon}
	\end{equation} in the convex hull of $\tilde S_2$ and $\tilde S_3$, which is $\tilde S_4.$
	
	Finally, dividing everything by $p(s,w)$, we obtain the following bound valid in $\tilde S_4$ and away from the poles of $A(s,w,z)$:
	\begin{equation}\label{Bound in vertical strips}
		\lvert A(s,w,z)\rvert \ll\lvert wz\rvert ^{\epsilon}(1+\lvert s\rvert )^{3+\epsilon}.
	\end{equation}

	\subsection{Residue of $A(s,w,z)$ at $s=3/2-w$}
	
	We have \begin{equation}\label{key}
		C(s,w,z;\psi,\psi')=\sum_{q\geq1}\frac{\psi'(q)}{q^s}\cdot D(w,z-w,q;\psi),
	\end{equation} and if $\psi=\psi_1,$ $q=\square,$ and $z-w\neq0$, $D(w,z-w,q;\psi)$ has a pole at $w=3/2$. Using the notation from the proof of Lemma \ref{Estimate For D(w,t)}, we have \begin{equation}\label{key}
		D(w,t,q;\psi_1)=\frac{L\lz w-1/2,\leg{4q}{\cdot}\psi_1\pz}{\zeta_{(4q)}(2w-1)}E(w,t,q;\psi_1)P_{p\mid q}(w,t,q;\psi_1).
	\end{equation}
	We now compute the residue. The residue at $w=3/2$ of $\frac{L\lz w-1/2,\leg{4q}{\cdot}\psi_1\pz}{\zeta_{(4q)}(2w-1)}$ is 
	\begin{equation}\label{key}
		\frac1{\zeta(2)}\prod_{p\mid 4q}\frac{p}{p+1}.
	\end{equation} 
	
	If $p$ is an odd prime and $p^{2a}\mid \mid q,$ then by $\eqref{Sound's Gauss sums - exact formula}$ \begin{equation}\label{key}
		\begin{aligned}
			P_{p\mid q}&\lz\frac32,z-\frac32,q;\psi_1\pz=\prod_{\substack{p\mid q,\\p\odd}}\sum_{k=0}^\infty\frac{G\lz\leg{\cdot}{p^k},q\pz\psi(p^k)a_{z-\frac32}(p^k)}{p^{\frac{3k}2}}\\
			&=\prod_{\substack{p\mid q,\\p\odd}}\lz1+\sum_{k=1}^a\frac{G\lz\leg{\cdot}{p^{2k}},q\pz a_{z-\frac32}(p^{2k})}{p^{3k}}+\frac{G\lz\leg{\cdot}{p^{2a+1}},q\pz a_{z-\frac32}(p^{2a+1})}{p^{3a+\frac32}}\pz\\
			&=\prod_{\substack{p\mid q,\\p\odd}}\lz1+\lz1-p^{\frac32-z}\pz\sum_{k=1}^{a}\frac{\phi(p^{2k})}{p^{3k}}+\lz 1-p^{\frac32-z}\pz\frac{p^{2a}\sqrt p}{p^{3a+\frac32}}\pz\\
			%&=\prod_{\substack{p\mid q,\\p\odd}}\lz1+(1-p^{-t})\lz\frac{p-1}{p}\sum_{k=1}^a p^{-k}+p^{-a-1}\pz\pz=\\
			%&=\prod_{\substack{p\mid q,\\p\odd}}\lz1+(1-p^{-t})\lz \frac{p-1}{p}\cdot\frac{1-p^{-a}}{p-1}+p^{-a-1}\pz\pz=\\
			&=\prod_{\substack{p\mid q,\\p\odd}}\lz1+\frac {1-p^{\frac32-z}}{p}\pz.
		\end{aligned}
	\end{equation}
	
	Finally, for $q=\square$ and $\psi=\psi_1$, we have
	\begin{equation}\label{E}
		\begin{aligned}
			&E\lz\frac32,z-\frac32,q;\psi\pz\\
			&=\frac1{L\lz z-\frac12,\leg{4q}{\cdot}\pz}\prod_{p}\lz 1+\frac{\leg{4q}{p}^2}{\lz p^{z-\frac12}-\leg {4q}p\pz\lz p+\leg{4q}{p}\pz}\pz\\
			&=\frac1{\zeta\lz z-\frac12\pz}\prod_{p\mid 4q}\lz1-\frac1{p^{z-\frac12}}\pz^{-1}\prod_{p\nmid 4q}\lz1+\frac1{\lz p^{z-\frac12}-1\pz\lz p+1\pz}\pz.
		\end{aligned}
	\end{equation}
	If we now denote \begin{equation}\label{key}
		P(z)=\prod_{p}\lz1+\frac1{\lz p^{z-\frac12}-1\pz\lz p+1\pz}\pz,
	\end{equation} the last expression in \eqref{E} equals
	\begin{equation}
		\begin{aligned}
			&\frac{ P(z)}{\zeta\lz z-\frac12\pz}\prod_{p\mid 4q}\lz1-\frac1{p^{z-\frac12}}\pz^{-1}\lz1+\frac1{\lz p^{z-\frac12}-1\pz(p+1)}\pz^{-1}\\
			%&=\frac {P(z)}{\zeta\lz z-\frac12\pz}\prod_{p\mid 4q}\frac{p^{z-1/2}}{p^{z-1/2}-1}\cdot\frac{\lz p^{z-\frac12}-1\pz(p+1)}{\lz p^{z-\frac12}-1\pz(p+1)+1}=\\
			&=\frac {P(z)}{\zeta\lz z-\frac12\pz}\prod_{p\mid 4q}\frac{p^{z-1/2}(p+1)}{\lz p^{z-\frac12}-1\pz(p+1)+1},
		\end{aligned}
	\end{equation}

	Putting everything together, we find that \begin{equation}\label{key}
		\begin{aligned}
			&\res_{w=3/2}C(s,w,z;\psi_1,\psi')=\frac{P(z)}{\zeta(2)\zeta\lz z-\frac12\pz}\sum_{q\geq1}\frac{\psi'(q^2)}{q^{2s}}\times\\
			&\times\prod_{p\mid 4q}\frac{p}{p+1}\cdot\frac{p^{z-1/2}(p+1)}{\lz p^{z-1/2}-1\pz(p+1)+1}\prod_{\substack{p\mid q,\\ p\odd}}\frac{p+1-p^{3/2-z}}{p}.
		\end{aligned}
	\end{equation}
	We have \begin{equation}\label{key}
		\frac p{p+1}\cdot\frac{p^{z-1/2}(p+1)}{(p^{z-1/2}-1)(p+1)+1}\cdot\frac{p+1-p^{3/2-z}}{p}=1,
	\end{equation} 
	so 
	\begin{equation}\label{key}
		\begin{aligned}
			\res_{w=3/2}C(s,w,z;\psi_1,\psi')&%=\frac{P(z)}{\zeta(2)\zeta\lz z-\frac12\pz}\cdot\frac{2^{z+1/2}}{3\cdot 2^{z-1/2}-2}\sum_{q\geq1}\frac{\psi'(q)^2}{q^{2s}}=\\
			&=\frac{P(z)}{\zeta(2)\zeta\lz z-\frac12\pz}\cdot\frac{2^{z+1/2}}{3\cdot 2^{z-1/2}-2}\cdot L\lz 2s,\psi'^2\pz.
		\end{aligned}
	\end{equation}
	
	Using \eqref{C(s,w,z) as twisted C(s,w,z)} gives
	\begin{equation}\label{Residue of C(s,w,z)}
		\begin{aligned}
			\res_{w=\frac32} C(s,w,z)&=4^{-s}\res_{w=\frac32}C(s,w,z;\psi_1,\psi_0)+\res_{w=\frac32}C(s,w,z;\psi_1,\psi_{-1})\\
			&=\frac{P(z)}{\zeta(2)\zeta(z-1/2)}\cdot\frac{2^{z+1/2}}{3\cdot 2^{z-1/2}-2}\lz 4^{-s}\zeta(2s)+L(2s,\psi_1)\pz\\
			&=\frac{P(z)\zeta(2s)}{\zeta(2)\zeta(z-1/2)}\cdot\frac{2^{z+1/2}}{3\cdot 2^{z-1/2}-2}.
		\end{aligned}
	\end{equation} Note that this is also true in the case $w=z,$ when there is no pole and the residue is $0$.
	
	The functional equation
	\begin{equation}\label{key}
		A(s,w,z)=\frac{\pi^{s-1/2}\Gamma\bfrac{1-s}{2}}{4^s\Gamma\bfrac s2}C(1-s,s+w,s+z)
	\end{equation}
	implies that
	\begin{equation}\label{key}
		\begin{aligned}
			&\res_{s=3/2-w}A(s,w,z)\\
			&=\frac{\pi^{1-w}\Gamma\bfrac{w-1/2}{2}}{\Gamma\bfrac{3/2-w}{2}}\cdot\frac{P(3/2-w+z)\zeta(2w-1)}{\zeta(2)\zeta(1-w+z)}\cdot\frac{2^{z+w-1}}{3\cdot 2^{z+1-w}-2}.
		\end{aligned}
	\end{equation}
	Substituting $w=1/2+\alpha$, $z=1/2+\beta$, we obtain
	\begin{equation}\label{key}
		\begin{aligned}
			&\res_{s=1-\alpha}A(s,1/2+\alpha,1/2+\beta)
			\\&=\frac{\pi^{1/2-\alpha}\Gamma\bfrac{\alpha}{2}}{\Gamma\bfrac{1-\alpha}{2}}\cdot\frac{P(3/2-\alpha+\beta)\zeta(2\alpha)}{\zeta(2)\zeta(1-\alpha+\beta)}\cdot\frac{2^{\alpha+\beta}}{3\cdot 2^{1-\alpha+\beta}-2}.
		\end{aligned}
	\end{equation}
	Using the functional equation \begin{equation}\label{key}
		\zeta(2\alpha)=\pi^{2\alpha-1/2}\frac{\Gamma\bfrac{1-2\alpha}{2}}{\Gamma(\alpha)}\zeta(1-2\alpha),
	\end{equation} the last expression becomes
	\begin{equation}\label{key}
		\frac{\pi^{\alpha}\Gamma\bfrac{\alpha}{2}\Gamma\bfrac{1-2\alpha}{2}}{\Gamma\bfrac{1-\alpha}{2}\Gamma(\alpha)}\cdot\frac{P(3/2-\alpha+\beta)\zeta(1-2\alpha)}{\zeta(2)\zeta(1-\alpha+\beta)}\cdot\frac{2^{\alpha+\beta}}{3\cdot 2^{1-\alpha+\beta}-2}.
	\end{equation}
	Now using the relation \eqref{RatioOfGammaFactors} gives
	\begin{equation}\label{Residue at s=3/2-w wrong gamma factors}
		\pi^{\alpha-\frac12}\cos\bfrac{\pi \alpha}{2}\Gamma\lz\frac12-\alpha\pz\frac{P\lz\frac32-\alpha+\beta\pz\zeta(1-2\alpha)}{\zeta(2)\zeta(1-\alpha+\beta)}\cdot\frac{2^{\beta}}{3\cdot 2^{-\alpha+\beta}-1},
	\end{equation}
	and an application of Lemma \ref{G odd + G even}
	yields
	\begin{equation}\label{Residue at s=3/2-w}
		\pi^{\alpha}\lz\Go\lz\frac12+\alpha\pz+\Ge\lz\frac12+\alpha\pz\pz\frac{P\lz\frac32-\alpha+\beta\pz\zeta(1-2\alpha)}{\zeta(2)\zeta(1-\alpha+\beta)}\cdot\frac{1}{6-2^{1+\alpha-\beta}}.
	\end{equation}
	\vspace{0pt}
	\section{Proof of Theorem \ref{Theorem for all characters}}
	We will now prove Theorem \ref{Theorem for all characters}. We have
	\begin{equation}\label{key}
		\sum_{\substack{n\geq1,\\n\odd}}\frac{L_{(2)}\lz\frac12+\alpha,\chi_n\pz}{L_{(2)}\lz\frac12+\beta,\chi_n\pz}f\bfrac nX=\frac1{2\pi i}\int\limits_{(2)}A\lz s,\frac12+\alpha,\frac12+\beta\pz\M f(s)X^sds.
	\end{equation}
	By \eqref{Final region of definition for A(s,w,z)} and the definition of $\tilde S_4$ in \eqref{Definition of S tilde}, we can shift the integral to $\re(s)=N(\alpha,\beta)+\epsilon,$ where 
	\begin{equation}\label{key}
		N(\alpha,\beta)=\max\left\{1-2\re(\alpha),1-2\re(\beta),\frac12-\re(\alpha),\frac12-\re(\beta),-\frac52\right\}.
	\end{equation} We capture the residues at $s=1$ \eqref{Residue at s=1} and $s=1-\alpha$ \eqref{Residue at s=3/2-w}, and then use \eqref{Bound in vertical strips} to estimate the error term. Thus if $\re(\alpha)>0$, $\re(\beta)>\epsilon$, we get
	\begin{equation}
		\begin{aligned}
			\sum_{\substack{n\geq1,\\n\odd}}&\frac{L_{(2)}(1/2+\alpha,\chi_n)}{L_{(2)}(1/2+\beta,\chi_n)}f(n/X)\\
			&=X\M f(1)\res_{s=1}A\lz s,\frac12+\alpha,\frac12+\beta\pz\\
			&+X^{1-\alpha}\M f(1-\alpha)\res_{s=1-\alpha}A\lz s,\frac12+\alpha,\frac12+\beta\pz\\
			&+\frac{1}{2\pi i}\int_{(N(\alpha,\beta)+\epsilon)}A(s,w,z)\M f(s)X^sds\\
			&=X\M f(1)\frac{\zeta_{(2)}(1+2\alpha)}{2\zeta_{(2)}(1+\alpha+\beta)}\prod_{p>2}\lz1+\frac{p^{\alpha-\beta}-1}{p^{1+\alpha-\beta}(p^{1+\alpha+\beta}-1)}\pz\\
			&+X^{1-\alpha}\M f(1-\alpha)\pi^{\alpha}\lz\Go\lz\frac12+\alpha\pz+\Ge\lz\frac12+\alpha\pz\pz\times\\
			&\times\frac{P\lz\frac32-\alpha+\beta\pz\zeta(1-2\alpha)}{\zeta(2)\zeta(1-\alpha+\beta)(6-2^{\alpha-\beta+1})}\\
			&+O\lz (1+\lvert \alpha\rvert )^{\epsilon}\lvert \beta\rvert ^{\epsilon}X^{N(\alpha,\beta)+\epsilon}\pz.
		\end{aligned}
	\end{equation}
	
	\section{Proof of Theorem \ref{Theorem for log derivatives}}
	
	Fix $r$ with $\re(r)>\epsilon$. We let 
	\begin{equation}\label{key}
		M_1(\alpha,\beta)=\M f(1)\frac{\zeta_{(2)}(1+2\alpha)}{2\zeta_{(2)}(1+\alpha+\beta)}\prod_{p>2}\lz1+\frac{p^{\alpha-\beta}-1}{p^{1+\alpha-\beta}(p^{1+\alpha+\beta}-1)}\pz,
	\end{equation} and 
	\begin{equation}\label{key}
		\begin{aligned}
			M_2(\alpha,\beta)=\M f(1-\alpha)\pi^{\alpha}&\lz\Go\lz\frac12+\alpha\pz+\Ge\lz\frac12+\alpha\pz\pz\times\\
			&\times\frac{P\lz\frac32-\alpha+\beta\pz\zeta(1-2\alpha)}{\zeta(2)\zeta(1-\alpha+\beta)(6-2^{\alpha-\beta+1})},
		\end{aligned}
	\end{equation} so that \eqref{Asymptotic for ratios of all characters} is
	\begin{equation}\label{key}
		\sum_{\substack{n\geq1,\\n\odd}}\frac{L_{(2)}(1/2+\alpha,\chi_n)}{L_{(2)}(1/2+\beta,\chi_n)}f(n/X)=XM_1(\alpha,\beta)+X^{1-\alpha}M_2(\alpha,\beta)+E(X,\alpha,\beta),
	\end{equation} where $E(X,\alpha,\beta)\ll X^{N(\alpha,\beta)+\epsilon}$ is the error term. Note that the left-hand side and both $M_1(\alpha,\beta)$ and $M_2(\alpha,\beta)$ are analytic functions of $\alpha,\beta$, so  $E(X,\alpha,\beta)$ is analytic too.
	
	To deduce Theorem \ref{Theorem for log derivatives}, we fix $\beta=r$ with $\re(\beta)>\epsilon$, differentiate with respect to $\alpha$, and set $\alpha=\beta=r.$
	
	For the first term, we get
	\begin{equation}\label{key}
		\begin{aligned}
			\frac{d}{d\alpha} XM_1(\alpha,\beta)\Big\vert_{\alpha=\beta=r}%&=X\M f(1)\lz\frac{\zeta_{(2)}'(1+2r)}{\zeta_{(2)}(1+2r)}-\frac{\zeta_{(2)}'(1+2r)}{2\zeta_{(2)}(1+2r)}+\sum_{p>2}\frac{\log p}{2p(p^{1+2r}-1)}\pz=\\
			&=\frac{X\M f(1)}{2}\lz\frac{\zeta_{(2)}'(1+2r)}{\zeta_{(2)}(1+2r)}+\sum_{p>2}\frac{\log p}{p(p^{1+2r}-1)}\pz.
		\end{aligned}
	\end{equation}
	For the second term, we notice that due to the factor $\frac{1}{\zeta(1-\alpha+\beta)}$, only one term survives. We also note that $P(3/2)=\zeta(2)$, so we get
	\begin{equation}\label{key}
		\begin{aligned}
			\frac{d}{d\alpha} &X^{1-\alpha}M_2(\alpha,\beta)\Big\vert_{\alpha=\beta=r}=\\
			&=-X^{1-r}\M f(1-r)\pi^r\lz\Go\lz\frac12+r\pz+\Ge\lz\frac12+r\pz\pz\frac{\zeta(1-2r)}{4}.
		\end{aligned}
	\end{equation}
	Therefore we have
	\begin{equation}\label{Sum of L'/L with removed 2-factors}
		\begin{aligned}
			&\sum_{\substack{n\geq1,\\n\odd}}\frac{L_{(2)}'(1/2+r,\chi_n)}{L_{(2)}(1/2+r,\chi_n)}f(n/X)=\\
			&=\frac{X\M f(1)}{2}\lz\frac{\zeta_{(2)}'(1+2r)}{\zeta_{(2)}(1+2r)}+\sum_{p>2}\frac{\log p}{p(p^{1+2r}-1)}\pz-\\
			&-X^{1-r}\M f(1-r)\pi^r\lz\Go\lz\frac12+r\pz+\Ge\lz\frac12+r\pz\pz\frac{\zeta(1-2r)}{4}+\\
			&+\frac{d}{d\alpha}E(X,\alpha,\beta)\Big\vert_{\alpha=\beta=r}.
		\end{aligned}
	\end{equation}
	Since $E(X,\alpha,\beta)$ is analytic in $\alpha$, we can use Cauchy's integral formula to compute its derivative. We have 
	\begin{equation}\label{key}
		\frac{d}{d\alpha}E(X,\alpha,\beta)=\frac{1}{2\pi i}\int_{C_\alpha}\frac{E(X,z,\beta)}{(z-\alpha)^2}dz,
	\end{equation} where $C_\alpha$ is a circle centered at $\alpha$ of radius $\rho$ with $\epsilon/2<\rho<\epsilon$. Then
	\begin{equation}\label{key}
		\lab\frac{d}{d\alpha}E(X,\alpha,\beta)\rab\ll \frac1{\rho}\cdot\max_{z\in C_\alpha} \lvert E(X,z,\beta)\rvert \ll (1+\lvert \alpha\rvert )^{\epsilon}\lvert \beta\rvert ^{\epsilon}X^{N(\alpha,\beta)+\epsilon}.
	\end{equation} Taking $\alpha=\beta=r$ and denoting $N(r):=N(r,r)$ gives
	\begin{equation}\label{key}
		\lab\frac{d}{d\alpha}E(X,\alpha,\beta)\rab\ll \lvert r\rvert ^{\epsilon}X^{N(r)+\epsilon}.
	\end{equation}

	Now we recover the Euler factors at 2 that we removed from $A(s,w,z)$. For odd $n$, we have \begin{equation}\label{key}
		L(1/2+r,\chi_n)=L_{(2)}(1/2+r,\chi_n)\lz1-\frac{\leg 2n}{2^{1/2+r}}\pz^{-1},
	\end{equation} so
	\begin{equation}\label{key}
		\begin{aligned}
			&\sum_{\substack{n\geq1,\\n\odd}}\frac{L'(1/2+r,\chi_n)}{L(1/2+r,\chi_n)}f(n/X)=\\
			&=\sum_{\substack{n\geq1,\\n\odd}}\frac{L_{(2)}'(1/2+r,\chi_n)}{L_{(2)}(1/2+r,\chi_n)}f(n/X)-\log 2\sum_{\substack{n\geq1,\\n\odd}}\frac{\leg2n}{2^{1/2+r}-\leg 2n}f(n/X).
		\end{aligned}
	\end{equation}
	Since for any $a\in\Z/8\Z$, partial summation gives 
	\begin{equation}\label{key}
		\sum_{n\equiv a\mod 8}f(n/X)=\frac{X}{8}\M f(1)+O(1),
	\end{equation} we get
	\begin{equation}\label{Sum of L'/L in terms of that with removed 2-factors}
		\begin{aligned}
			&\sum_{\substack{n\geq1,\\n\odd}}\frac{L'(1/2+r,\chi_n)}{L(1/2+r,\chi_n)}f(n/X)=\\
			%&=\sum_{n\odd}\frac{L_{(2)}'(1/2+r,\chi_n)}{L_{(2)}(1/2+r,\chi_n)}f(n/X)-\frac{X\M f(1)}{4}\log 2\lz\frac1{2^{1/2+r}-1}-\frac{1}{2^{1/2+r}+1}\pz+O(1)=\\
			&=\sum_{\substack{n\geq1,\\n\odd}}\frac{L_{(2)}'(1/2+r,\chi_n)}{L_{(2)}(1/2+r,\chi_n)}f(n/X)-\frac{X\M f(1)}{2}\cdot\frac{\log 2}{2^{1+2r}-1}+O(1).
		\end{aligned}
	\end{equation}
	
	We also have
	\begin{equation}\label{zeta'/zeta with removed 2-factors}
		\frac{\zeta'(1+2r)}{\zeta(1+2r)}=\frac{\zeta_{(2)}'(1+2r)}{\zeta_{(2)}(1+2r)}-\frac{\log 2}{2^{2r+1}-1},
	\end{equation}
	so using \eqref{Sum of L'/L with removed 2-factors} and \eqref{zeta'/zeta with removed 2-factors} in \eqref{Sum of L'/L in terms of that with removed 2-factors}, we obtain
	\begin{equation}
		\begin{aligned}
			&\sum_{\substack{n\geq1,\\n\odd}}\frac{L'(1/2+r,\chi_n)}{L(1/2+r,\chi_n)}f(n/X)=\\
			&=\frac{X\M f(1)}{2}\lz\frac{\zeta'(1+2r)}{\zeta(1+2r)}+\sum_{p>2}\frac{\log p}{p(p^{1+2r}-1)}\pz-\\
			&-X^{1-r}\M f(1-r)\pi^r\lz\Go\lz\frac12+r\pz+\Ge\lz\frac12+r\pz\pz\frac{\zeta(1-2r)}{4}+\\
			&+O\lz1+ \lvert r\rvert ^{\epsilon}X^{N(r)+\epsilon}\pz.
		\end{aligned}
	\end{equation}
	
	\section{Proof of Theorem \ref{Theorem for log derivatives for squarefree moduli}}
	
	In this section, we assume that $\epsilon<\re(r)<1/4$. We write all odd integers $n$ as $n=n_0n_1^2$ with $n_0$ square-free. Then we have
	\begin{equation}\label{sieving log derivatives for squarefree}
		\begin{aligned}
			\sum_{\substack{n\geq1,\\n\odd}}&\frac{\mu^2(n)L'(1/2+r,\chi_n)}{L(1/2+r,\chi_n)}f(n/X)=\\
			%&=\sum_{n\odd}\frac{L'(1/2+r,\chi_n)}{L(1/2+r,\chi_n)}f(n/X)\sum_{d^2\mid n}\mu(d)=\\
			&=\sum_{\substack{d\geq1,\\d\odd}}\mu(d)\sum_{\substack{n\geq1,\\n\odd}}\frac{L'(1/2+r,\chi_{nd^2})}{L(1/2+r,\chi_{nd^2})}f\bfrac{nd^2}{X}.
		\end{aligned}
	\end{equation}
	From
	\begin{equation}\label{key}
		L(s,\chi_{nd^2})=L(s,\chi_{n})\prod_{p\mid d}\lz1-\frac{\chi_{n}(p)}{p^s}\pz,
	\end{equation} we obtain
	\begin{equation}\label{key}
		\frac{L'(1/2+r,\chi_{nd^2})}{L(1/2+r,\chi_{nd^2})}=\frac{L'(1/2+r,\chi_n)}{L(1/2+r,\chi_n)}+\sum_{p\mid d}\frac{\chi_n(p)\log p}{p^{1/2+r}-\chi_n(p)},
	\end{equation}
	so \eqref{sieving log derivatives for squarefree} equals
	\begin{equation}\label{Sieved log derivatives}
		\begin{aligned}
			&\sum_{\substack{d\geq1,\\d\odd}}\mu(d)\sum_{\substack{n\geq1,\\n\odd}}\frac{L'(1/2+r,\chi_n)}{L(1/2+r,\chi_n)}f\bfrac{nd^2}{X}+\\
			&+\sum_{\substack{d\geq1,\\d\odd}}\mu(d)\sum_{\substack{n\geq1,\\n\odd}}f\bfrac{nd^2}{X}\sum_{p\mid d}\frac{\chi_n(p)\log p}{p^{1/2+r}-\chi_n(p)}.
		\end{aligned}
	\end{equation}
	
	Let us now consider the first sum. The terms with $d> \sqrt X$ here contribute $\ll \lvert r\rvert ^\epsilon X^{1/2}$ by \eqref{Size of log derivative}. For the other terms, we use \eqref{Asymptotic formula for ratios of log derivative} (with $X/d^2$ instead of $X$). Writing the right-hand side of \eqref{Asymptotic formula for ratios of log derivative} as $X M_1(r)-X^{1-r} M_2(r)+E(X,r)$, this part of \eqref{Sieved log derivatives} contributes
	\begin{equation}\label{Main term for squarefrees}
		\begin{aligned}
			&XM_1(r)\sum_{\substack{d\leq\sqrt X,\\d\odd}}\frac{\mu(d)}{d^2}-X^{1-r}M_2(r)\sum_{\substack{d\leq\sqrt X,\\d\odd}}\frac{\mu(d)}{d^{2-2r}}+\sum_{\substack{d\leq\sqrt X,\\d\odd}}\mu(d)E\lz\frac X{d^2},r\pz\\
			&=\frac{4X M_1(r)}{3\zeta(2)}-\frac{X^{1-r} M_2(r)}{\zeta_{(2)}(2-2r)}+O(\lvert r\rvert ^\epsilon X^{1-2r+\epsilon}).
		\end{aligned}
	\end{equation}
	
	Now we compute the second term in \eqref{Sieved log derivatives}. Setting $s=\frac12+r$ and exchanging summations gives
	\begin{equation}\label{Error term for squarefree log derivatives}
		\begin{aligned}
			&\sum_{p\odd}\frac{\log p}{p^s}\sum_{\substack{d\geq1,\\d\odd}}\mu(pd)\sum_{\substack{n\geq1,\\n\odd}}\frac{\chi_n(p)}{1-\chi_n(p)/p^s}f\bfrac{nd^2p^2}{X}\\
			&=\sum_{p\odd}\frac{\log p}{p^s}\sum_{\substack{d\geq1,\\d\odd}}\mu(pd)\sum_{\substack{n\geq1,\\n\odd}}\chi_n(p)f\bfrac{nd^2p^2}{X}\sum_{k=0}^\infty \frac{\chi_n(p)^k}{p^{ks}}.
		\end{aligned}
	\end{equation}
	We split the last sum depending on the parity of $k$. 
	
	For an even $k$, the sum equals \begin{equation}\label{contribution of even k}
		\sum_{p\odd}\frac{\log p}{p^{(k+1)s}}\sum_{\substack{d\geq1,\\d\odd}}\mu(pd)\sum_{\substack{n\geq1,\\n\odd}}\chi_n(p)f\bfrac{nd^2p^2}{X}.
	\end{equation} 
	If we denote \begin{equation}\label{key}
		\tilde\chi_p=\begin{cases}
			\chi_p&\hbox{if $p\equiv1\mod 4$,}\\
			\chi_p\psi_{-1}&\hbox{if $p\equiv3\mod 4$},
		\end{cases}
	\end{equation} the inner sum is \begin{equation}\label{Even k 1}
		\frac1{2\pi i}\int_{(c)}\frac{L_{(2)}(u,\tilde\chi_p)\M f(u)X^u}{d^{2u}p^{2u}}du\ll\frac{X^{1/2+\epsilon}}{d^{1+2\epsilon}p^{1+\epsilon}}
	\end{equation} where we shifted the integral to $c=1/2+\epsilon$ and used Lindel\"of's bound \eqref{Lindelof}. 
	
	Using \eqref{Even k 1} and summing \eqref{contribution of even k} over all even $k\geq0$, we see that the contribution of even $k$ in \eqref{Error term for squarefree log derivatives} is \begin{equation}\label{key}
		\ll X^{1/2+\epsilon}\sum_{k=0}^{\infty}\sum_{p\geq 1}\frac{\log p}{p^{(2k+1)s+1+\epsilon}}\sum_{d\geq 1}\frac1{d^{1+2\epsilon}}\ll X^{1/2+\epsilon}.
	\end{equation}
	
	Now we compute the contribution of odd values of $k$ into \eqref{Error term for squarefree log derivatives}. For an odd $k$, the summand equals
	\begin{equation}\label{key}
		\begin{aligned}
			&\sum_{p>2}\frac{\log p}{p^{s}}\sum_{\substack{d\geq1,\\d\odd}}\mu(pd)\sum_{\substack{n\geq1,\\n\odd}}\frac{\chi_n(p)^{k+1}}{p^{ks}}f\bfrac{nd^2p^2}{X}\\
			&=\sum_{p>2}\frac{\log p}{p^{(k+1)s}}\sum_{\substack{d\geq1,\\d\odd}}\mu(pd)\sum_{\substack{\substack{n\geq1,\\n\odd},\\p\nmid n}}f\bfrac{nd^2p^2}{X}.
		\end{aligned}
	\end{equation}
	We write the innermost sum as a Mellin integral:
	\begin{equation}\label{key}
		\begin{aligned}
			\sum_{\substack{\substack{n\geq1,\\n\odd,}\\p\nmid n}}f\bfrac{nd^2p^2}{X}&=\frac{1}{2\pi i}\int_{(c)}\frac{\zeta_{(2p)}(u)X^{u}\M f(u)}{d^{2u}p^{2u}}du\\
			&=\frac{X\M f(1)(1-1/p)}{2d^2p^2}+O\lz\frac{X^{1/2+\epsilon}}{d^{1+2\epsilon}p^{1+2\epsilon}}\pz,
		\end{aligned}
	\end{equation} where the last equation holds after shifting the integral to $c=1/2+\epsilon$. Summing the error term over all odd $k$, $d$ and $p$ contributes $\ll X^{1/2+\epsilon}$, so it remains to compute the contribution of the main term. This gives
	\begin{equation}\label{key}
		\begin{aligned}
			&\frac{X\M f(1)}{2}\sum_{p>2}\frac{\lz1-\frac1p\pz\log p}{p^{2+(k+1)s}}\sum_{\substack{d\geq1,\\d\odd}}\frac{\mu(pd)}{d^2}=\frac{-2X\M f(1)}{3\zeta(2)}\sum_{p>2}\frac{\log p}{(p+1)p^{1+(k+1)s}}.
		\end{aligned}
	\end{equation}
	Summing this over odd values of $k$ gives
	\begin{equation}\label{Extra main term for squarefrees}
		\begin{aligned}
			&\frac{-2X\M f(1)}{3\zeta(2)}\sum_{p>2}\frac{\log p}{(p+1)p^{1+s}}\sum_{k=0}^{\infty}\frac1{p^{(2k+1)s}}=\frac{-2X\M f(1)}{3\zeta(2)}\sum_{p>2}\frac{\log p}{p(p+1)(p^{2s}-1)}
		\end{aligned}
	\end{equation}
	
	Putting all together, we obtain the final result
	\begin{equation}\label{Asymptotic for log derivatives for squarefrees}
		\begin{aligned}
			&\sum_{\substack{n\geq1,\\n\odd}}\frac{\mu^2(n)L'(1/2+r,\chi_n)}{L(1/2+r,\chi_n)}f(n/X)\\
			%&=\frac{2X\M f(1)}{3\zeta(2)}\lz\frac{\zeta'(1+2r)}{\zeta(1+2r)}+\sum_{p>2}\frac{\log p}{p(p^{1+2r}-1)}-\sum_{p>2}\frac{\log p}{p(p+1)(p^{1+2r}-1)}\pz-\\
			%&-X^{1-r}\M f(1-r)\pi^r\lz\Go\lz\frac12+r\pz+\Ge\lz\frac12+r\pz\pz\frac{\zeta(1-2r)}{4\zeta_{(2)}(2-2r)}+\\\\
			%&+O\lz \mid r\mid ^\epsilon X^{1-2r+\epsilon}\pz=\\
			&=\frac{2X\M f(1)}{3\zeta(2)}\lz\frac{\zeta'(1+2r)}{\zeta(1+2r)}+\sum_{p>2}\frac{\log p}{(p+1)(p^{1+2r}-1)}\pz\\
			&-X^{1-r}\M f(1-r)\pi^r\lz\Go\lz\frac12+r\pz+\Ge\lz\frac12+r\pz\pz\frac{\zeta(1-2r)}{4\zeta_{(2)}(2-2r)}\\
			&+O\lz \lvert r\rvert ^\epsilon X^{1-2r+\epsilon}\pz.
		\end{aligned}
	\end{equation}
	
	\section{Proof of Corollary \ref{Corollary one-level density}}\label{Section one-level density}
	In this section, we show how the one-level density can be computed using the ratios conjecture. We follow Section 3 of \cite{CoSn}.
	
	Let $h(x)$ be an even Schwartz function whose Fourier transform $\hat h$ is supported in the interval $[-a,a]$ for some $a>0$. It follows that $h$ has an analytic continuation to the whole $\C$ via Fourier inversion.
	
	Recall that the one-level density is defined by 
	\begin{equation}\label{key}
		D(X;h)=\frac{1}{F(X)}\sum_{\substack{n\geq1\\n\odd}}\mu^2(n)f\bfrac nX\sum_{\gamma_n}h\bfrac{\gamma_n\log X}{2\pi},
	\end{equation} where $\gamma_n$ runs over the imaginary parts of the non-trivial zeros of $L(s,\chi_n)$, and
	\begin{equation}\label{Size of family}
		\begin{aligned}
			F(X)&=\sum_{\substack{n\geq 1,\\n\odd}}\mu^2(n)f\bfrac{n}{X}=\frac1{2\pi i}\int_{(c)}\frac{\zeta_{(2)}(s)}{\zeta_{(2)}(2s)}X^s\M f(s)ds\\
			&= \frac{2X\M f(1)}{3\zeta(2)}+O(X^{1/4+\epsilon}).
		\end{aligned}
	\end{equation}

	\smallskip
	
	By the residue theorem, we have
	\begin{equation}\label{}
		\begin{aligned}
			&F(X)D(X;h)\\
			&=\frac{1}{2\pi i}\lz\int_{(c)}-\int_{(1-c)}\pz h\lz\frac{\log X}{2\pi i}(s-1/2)\pz\sum_{\substack{n\geq1,\\n\odd}}\mu^2(n)f(n/X)\frac{L'(s,\chi_n)}{L(s,\chi_n)} ds
		\end{aligned}
	\end{equation} for any $1/2<c<1.$
	
	Using the functional equation \eqref{Functional equation for L-functions of primitive characters} and the fact that $h$ is even, the integral over $(1-c)$ equals
	%	\frac{L'(1-s,\chi_n)}{L(1-s,\chi_n)}=\log(\pi/n)+\frac{G'_n(s)}{G_n(s)}-\frac{L'(s,\chi_n)}{L(s,\chi_n)},
	%\end{equation} so
	\begin{equation}\label{key}
		\begin{aligned}
			&\frac{1}{2\pi i}\int_{(c)}h\lz\frac{\log X}{2\pi i}\lz s-\frac12\pz\pz\times\\
			&\times\lz\sum_{\substack{n\geq 1,\\n\odd}}\mu^2(n)f\bfrac nX\lz\log\bfrac\pi{n}+\frac{\Gamma_n^{'}}{\Gamma_n}(1-s)-\frac{L'}{L}(s,\chi_n)\pz\pz ds,
		\end{aligned}
	\end{equation} where $\Gamma_n(s)$ is $\Ge(s)$ or $\Go(s),$ depending on the parity of $\chi_n$.
	
	Hence
	\begin{equation}\label{One-level density as integral}
		\begin{aligned}
			&F(X)D(X;h)=\\
			&=\frac{1}{\pi i}\int_{(c)}h\lz\frac{\log X}{2\pi i}(s-1/2)\pz\sum_{\substack{n\geq1,\\n\odd}}\mu^2(n)f(n/X)\frac{L'(s,\chi_n)}{L(s,\chi_n)} ds\\
			&-\frac{1}{2\pi i}\int_{(c)}h\lz\frac{\log X}{2\pi i}(s-1/2)\pz\sum_{\substack{n\geq1,\\n\odd}}\mu^2(n)f(n/X)\lz\log\bfrac \pi n+\frac{\Gamma_n'(1-s)}{\Gamma_n(1-s)}\pz ds.
		\end{aligned}
	\end{equation}
	
	We now consider the first integral.
	We start with the following lemma:
	\begin{lemma}\label{Lemma size of h}
		For $h$ as above, and any integer $A>0$, we have 
		\begin{equation}
			h\bfrac{s\log X}{2\pi i}\ll_A \frac{X^{a\cdot\re(s)}}{\lvert s\rvert ^A(\log X)^A}.
		\end{equation}
	\end{lemma}
	\begin{proof}
		Using Fourier inversion and integrating by parts $A$-times, we obtain
		\begin{equation}\label{key}
			\begin{aligned}
				h\bfrac{s\log X}{2\pi i}&=\int_{-\infty}^{\infty}\hat h(t)e^{ts\log X}dt\\
				&=\frac{1}{s^A(\log X)^A}\int_{-\infty}^\infty \hat h^{(A)}(t)X^{ts}dt\\
				&\ll_A\frac{X^{a\cdot\re(s)}}{\lvert s\rvert ^A(\log X)^A}.
			\end{aligned}
		\end{equation}
	\end{proof}
	
	We now split the integral, writing
	\begin{equation}\label{key}
		\frac{1}{\pi i}\int_{(c)}h\lz\frac{\log X}{2\pi i}(s-1/2)\pz\sum_{\substack{n\geq1,\\n\odd}}\mu^2(n)f(n/X)\frac{L'(s,\chi_n)}{L(s,\chi_n)} ds=: I_1+I_2,
	\end{equation} where $I_1$ is the part of the integral with $\lvert \im(s)\rvert \leq X$, and $I_2$ the part with $\lvert \im(s)\rvert >X$.
	
	For $I_2$, we use Lemma \ref{Lemma size of h} and \eqref{Size of log derivative}, and obtain the bound
	\begin{equation}\label{key}
		\begin{aligned}
			&\sum_{\substack{n\geq1,\\n\odd}}\mu^2(n)f\bfrac nX \log^2(n)	\int_{\substack{\re(s)=c\\\lvert \im(s)\rvert >X}}\log ^2(\lvert s\rvert ) \lab h\lz\frac{\log X}{2\pi i}\lz s-1/2\pz\pz\rab ds\\
			&\ll_A X^{1+\epsilon+a(c-1/2)}\int_{\lvert t\rvert >X}\frac{\log ^2(\lvert t\rvert )}{t^A} dt\ll X^{\epsilon},
		\end{aligned}
	\end{equation} where we took any $A>a(c-1/2)+1.$

	For $I_1$, we substitute $s=1/2+r+it$ with $0<r<1/4$ and use Theorem \ref{Theorem for log derivatives for squarefree moduli}. Thus $I_1$ is
	\begin{equation}\label{Integral after ratios}
		\begin{aligned}
			\frac{X}{\pi}&\int_{-X}^{X} h\lz\frac{\log X}{2\pi i}(r+it)\pz\Bigg\{\\
			&\frac{2\M f(1)}{3\zeta(2)}\lz\frac{\zeta'(1+2r+2it)}{\zeta(1+2r+2it)}+\sum_{p>2}\frac{\log p}{(p+1)(p^{1+2r+2it}-1)}\pz\\
			&-\frac{\M f(1-r-it)\pi^{r+it}\lz\Go\lz\frac12+r+it\pz+\Ge\lz\frac12+r+it\pz\pz\zeta(1-2r-2it)}{X^{r+it}4\zeta_{(2)}(2-2r-2it)}\\
			&+O(\lvert t\rvert ^{\epsilon}X^{1-2r+\epsilon})
			\Bigg\}dt.
		\end{aligned}
	\end{equation}
	Using Lemma \ref{Lemma size of h}, we can bound the error term by \begin{equation}\label{Error in one-level density}
		\begin{aligned}
			& X^{1-2r+\epsilon}\int_{-X}^{X}h\lz\frac{\log X}{2\pi i}(r+it)\pz\lvert t\rvert ^{\epsilon}dt\ll X^{1-2r+ar+\epsilon},
		\end{aligned}
	\end{equation} and this is $o(X)$ as long as $a<2$. Setting $r=1/4-\epsilon$ gives the error term in Corollary \ref{Corollary one-level density}. 
	
	We remark here that this is the only restriction on $a$, so any improvement of the error term in Theorem \ref{Theorem for log derivatives for squarefree moduli} would allow us to extend the support of the Fourier transform of our test functions.
	
	Now we compute the main terms of \eqref{Integral after ratios}.
	Since the integrand is regular at $r+it=0$, we can shift the line of integration to $r=0$, where the contribution of the horizontal integrals is by \eqref{Size of log derivative} and Lemma \ref{Lemma size of h}
	\begin{equation}\label{key}
		\ll X^{1+\epsilon}\int_0^{1/4}\left| h\lz\frac{\log X}{2\pi i}(r\pm iX)\pz\right|\ll_A X^{1+\epsilon+a/4-A}.
	\end{equation} We extend the range of integration from $-\infty$ to $\infty$ introducing an error of size $X^{\epsilon}$, and get
	\begin{equation}\label{Main term in one-level density}
		\begin{aligned}
			&\frac X{\pi}\int_{-\infty}^{\infty}h\lz\frac{\log X}{2\pi}t\pz\Bigg\{\frac{2\M f(1)}{3\zeta(2)}\lz\frac{\zeta'(1+2it)}{\zeta(1+2it)}\pz+\sum_{p>2}\frac{\log p}{(p+1)(p^{1+2it}-1)}\\
			&-\frac{X^{-it}\M f(1-it)\pi^{it}\lz\Go(1/2+it)+\Ge(1/2+it)\pz\zeta(1-2it)}{4\zeta_{(2)}(2-2it)}
			\Bigg\}dt\\
			&\hspace{-15pt}=\frac{2X}{\log X}\int_{-\infty}^{\infty}h(u)\Bigg\{\frac{2\M f(1)}{3\zeta(2)}\lz\frac{\zeta'}{\zeta}\lz1+\frac{4\pi i u}{\log X}\pz+\sum_{p>2}\frac{\log p}{(p+1)(p^{1+\frac{4\pi iu}{\log X}}-1)}\pz\\	
			&-\frac{\M f\lz1-\frac{2\pi i u}{\log X}\pz\pi^{\frac{2\pi i u}{\log X}}\lz\Go\lz\frac{1}{2}+\frac{2\pi i u}{\log X}\pz+\Ge\lz\frac{1}{2}+\frac{2\pi i u}{\log X}\pz\pz\zeta\lz1-\frac{4\pi i u}{\log X}\pz}{e^{2\pi i u}\cdot4\zeta_{(2)}\lz2-\frac{4\pi i u}{\log X}\pz}
			\Bigg\}du.
		\end{aligned}
	\end{equation} 
	
	We now compute the second integral in \eqref{One-level density as integral}. We shift the line of integration to $c=1/2$ and substitute $\frac{\log X}{2\pi i}\lz s-\frac12\pz=u$, getting \begin{equation}\label{Second integral in one-level density}
		\frac{1}{\log X}\int_{-\infty}^{\infty}h(u)\sum_{\substack{n\geq 1,\\n\odd}}\mu^2(n)f\bfrac nX\lz\log\bfrac \pi n-\frac{\Gamma_n'}{\Gamma_n}\lz\frac12-\frac{2\pi i u}{\log X}\pz\pz du.
	\end{equation} 
	The first part of Corollary \ref{Corollary one-level density} now follows from  \eqref{Error in one-level density} with $r=1/4-\epsilon$, \eqref{Main term in one-level density}, \eqref{Second integral in one-level density} and noticing that a similar computation as in Lemma \ref{Lemma character sum} gives
	\begin{equation}\label{average of gamma factors}
		\begin{aligned}
			&\sum_{\substack{n\geq1,\\n\odd}}\frac{\Gamma'_n(1-s)}{\Gamma_n(1-s)}\mu^2(n)f(n/X)=\\
			&=\frac{\Ge'(1-s)}{\Ge(1-s)}\sum_{n\equiv1\mod 4}\mu^2(n)f(n/X)+\frac{\Go'(1-s)}{\Go(1-s)}\sum_{n\equiv 3\mod 4}\mu^2(n)f(n/X)=\\
			&=\lz\frac{\Ge'(1-s)}{\Ge(1-s)}+\frac{\Go'(1-s)}{\Go(1-s)}\pz\lz\frac{X\M f(1)}{3\zeta(2)}+O(\sqrt X)\pz.
		\end{aligned}
	\end{equation}
	
	To prove the second part, we use the Laurent expansions $\zeta(s)=\frac1{s-1}+\dots$, $\frac{\zeta'}{\zeta}(s)=\frac{-1}{s-1}+\dots$, so \eqref{Main term in one-level density} is 
	\begin{equation}\label{First integral in one-level density - result}
		\frac{2 X \M f(1)}{3\zeta(2)}\int_{-\infty}^{\infty}h(u)\lz\frac{e^{-2\pi iu}-1}{2\pi iu}\pz du+O\bfrac{X}{\log X}.
	\end{equation}

	We have \begin{equation}\label{key}
		\sum_{\substack{n\geq 1,\\n\odd}}\mu^2(n)f\bfrac nX\log\bfrac \pi n=\frac{-2 \M f(1)X\log X}{3\zeta(2)}+O(X),
	\end{equation}
	so by \eqref{average of gamma factors}, \eqref{Second integral in one-level density} is
	\begin{equation}\label{Second integral in one-level density - result}
		-\frac{2 X\M f(1)}{3\zeta(2)}\int_{-\infty}^\infty h(u)du+ O\bfrac{X}{\log X}.
	\end{equation}
	
	Altogether, we obtain
	\begin{equation}\label{key}
		D(X;h)=\int_{-\infty}^{\infty} h(u)\lz1+\frac{e^{-2\pi i u}-1}{2\pi i u}\pz du+O\bfrac1{\log X}.
	\end{equation}
	
	Finally, to see that this result agrees with the density conjecture of Katz and Sarnak, we use the fact that $h(u)$ is even. We can thus drop the last term in the integral and replace $\frac{e^{-2\pi iu}}{2\pi iu}$ by \begin{equation}\label{key}
		\frac{1}{2}\lz\frac{e^{-2\pi iu}-e^{2\pi iu}}{2\pi iu}\pz=\frac{-\sin(2\pi u)}{2\pi u},
	\end{equation} which completes the proof.
	\appendix
	\section{\bf{Proof of the functional equation from Proposition \ref{Functional equation with Gauss sums}}}
	We now give the proof of the functional equation in Proposition \ref{Functional equation with Gauss sums}. We only give the details in the case of even characters, and then explain the usual modification for odd characters. For a detailed proof of the classical functional equation, see for example \cite{Gar}.
	
	Let $\chi$ be an even character modulo $q$. We define two theta functions
	\begin{equation}\label{key}
		\theta_\chi(y)=\sum_{n\in\Z}\chi(n)e^{-\pi n^2y},
	\end{equation} and 
	\begin{equation}\label{key}
		\theta_{\tau(\chi)}(y)=\sum_{n\in\Z}\tau(\chi,n)e^{-\pi n^2y}.
	\end{equation}
	Then using the fact that $\chi(n)$ is even, we have the integral representation of the L-function:
	\begin{equation}\label{key}
		\int_{0}^{\infty}y^{\frac s2}\frac{\theta_\chi(y)}{2}\frac{dy}{y}=\sum_{n\geq1}\chi(n)\int_{0}^{\infty}y^{\frac s2}e^{-\pi n^2y}\frac{dy}{y}=\pi^{-\frac s2}\Gamma\bfrac s2 L(s,\chi),
	\end{equation} and similarly
	\begin{equation}\label{key}
		\int_{0}^{\infty}y^{\frac s2}\frac{\theta_{\tau(\chi)}(y)}{2}\frac{dy}{y}=%\sum_{n\geq1}\tau(\chi,n)\int_{0}^{\infty}y^{\frac s2}e^{-\pi n^2y}\frac{dy}{y}=
		\pi^{-\frac s2}\Gamma\bfrac s2 K(s,\chi).
	\end{equation}
	These integrals converge absolutely for $s$ with $\re(s)>0.$
	
	The two theta functions are related by the following functional equation, which follows after an application of the Poisson summation.
	\begin{lemma}
		Let $\chi$ be a character modulo $q$. Then
		\begin{equation}\label{functional equation for theta}
			\theta_{\chi}(y)=\frac{1}{q\sqrt y}\theta_{\tau(\chi)}\bfrac{1}{yq^2}.
		\end{equation}
	\end{lemma}
	\begin{proof}
		We have 
		\begin{equation}\label{key}
			\theta_{\chi}(y)=\sum_{n\in\Z}\chi(n)e^{-\pi n^2y}=\sum_{j\mod q}\chi(j)\sum_{n\in\Z}e^{-\pi(qn+j)^2y}.
		\end{equation}
		The Fourier transform of the function in the inner sum is
		\begin{equation}\label{key}
			\begin{aligned}
				\int_{-\infty}^{\infty} e^{-\pi(qt+j)^2y}e^{-2\pi itn}dt&=\frac{1}{q\sqrt y}\int_{-\infty}^{\infty}e^{-\pi u^2}e^{-2\pi in\lz\frac{u\sqrt y}{q}-\frac jq\pz}du\\
				&=\frac{e^{2\pi i jn/q}}{q\sqrt y}\cdot e^{\frac{-\pi n^2}{yq^2}},
			\end{aligned}
		\end{equation}
		so Poisson summation gives
		\begin{equation}\label{key}
			\begin{aligned}
				\theta{\chi}(y)&=\frac{1}{q\sqrt y}\sum_{j\mod q}\chi(j)\sum_{n\in\Z}e^{2\pi i jn/q}\cdot e^{\frac{-\pi n^2}{yq^2}}\\
				&=\frac{1}{q\sqrt y}\sum_{n\in\Z}e^{\frac{-\pi n^2}{yq^2}}\sum_{j\mod q}\chi(j)e^{2\pi ijn/q}\\
				&=\frac{1}{q\sqrt y}\sum_{n\in\Z}\tau(\chi,n)e^{\frac{-\pi n^2}{yq^2}}\\
				&=\frac{1}{q\sqrt y}\theta_{\tau(\chi)}\bfrac{1}{yq^2}.	
			\end{aligned}
		\end{equation}
	\end{proof}
	Substituting $y\mapsto\frac{1}{yq^2}$ and using relation \eqref{functional equation for theta}, we obtain
	\begin{equation}\label{key}
		\begin{aligned}
			\pi^{-\frac s2}\Gamma\bfrac s2 L(s,\chi)&=\int_{0}^{\infty}y^{\frac s2}\frac{\theta_\chi(y)}{2}\frac{dy}{y}\\
			&=\int_{0}^{\infty}\bfrac1{yq^2}^{s/2}\frac{\theta_{\chi}\bfrac{1}{yq^2}}{2}\frac{dy}{y}\\
			&=q^{-s}\int_{0}^{\infty}y^{\frac{1-s}{2}}\frac{\theta_{\tau(\chi)}(y)}{2}\frac{dy}{y}\\
			&=\frac{\pi^{\frac{s-1}{2}}}{q^s}\Gamma\bfrac{1-s}{2}K(1-s,\chi),
		\end{aligned}
	\end{equation}
	which holds for $s$ with $0<\re(s)<1$. However, since the left-hand side has a meromorphic continuation to the whole complex plane, the functional equation holds for all $s\in \C$. 
	This finishes the proof for even characters. 
	
	For odd characters, the functions $\theta_\chi(y),\theta_{\tau(\chi)}(y)$ are identically 0, so we work with the following functions instead:
	\begin{equation}\label{Theta function for odd characters}
		\tilde\theta_\chi(y)=\sum_{n\in\Z}n\chi(n)e^{-\pi n^2y},
	\end{equation} and
	\begin{equation}\label{Theta function for odd characters with Gauss sums}
		\tilde\theta_\chi(y)=\sum_{n\in\Z}n\tau(\chi,n)e^{-\pi n^2y}.
	\end{equation} The sign in the resulting functional equation comes from the fact that the function $f(x)=xe^{-\pi x^2}$ is an eigenfunction of the Fourier transform with eigenvalue $-i$.

	\section{\bf{Ratios conjecture predictions for the family $\leg{\cdot}{n}$ for odd square-free $n$}}
	In this section, we follow the recipe of Conrey, Farmer and Zirnbauer for the family of Dirichlet characters $\chi_n=\leg \cdot n$ with $n$ odd and square-free, which is slightly different from the family of $\leg{d}{\cdot}$ for positive fundamental discriminants $d$ usually considered in the literature. In our case, the characters $\chi_n$ are primitive and even if $n\equiv1\mod 4$, or odd if $n\equiv3\mod 4$. We therefore split the family according to the parity of the character as
	\begin{equation}\label{Ratios}
		\begin{aligned}
			&\sum_{\substack{n\leq X,\\n \odd}}\frac{\mu^2(n)L(1/2+\alpha,\chi_n)}{L(1/2+\beta,\chi_n)}\\
			&=\sum_{\substack{n\leq X,\\ n\equiv1\mod 4}}\frac{\mu^2(n)L(1/2+\alpha,\chi_n)}{L(1/2+\beta,\chi_n)}+\sum_{\substack{n\leq X,\\ n\equiv3\mod 4}}\frac{\mu^2(n)L(1/2+\alpha,\chi_n)}{L(1/2+\beta,\chi_n)}.
		\end{aligned}
	\end{equation} The approximate functional equation is 
	\begin{equation}\label{key}
		L(1/2+\alpha,\chi_n)\approx\sum_{m\leq x}\frac{\chi_n(m)}{m^{1/2+\alpha}}+\bfrac{n}{\pi}^{-\alpha}\Gamma_{e/o}(1/2+\alpha)\sum_{m\leq y}\frac{\chi_n(m)}{m^{1/2-\alpha}},
	\end{equation} where $xy\approx n/2\pi$, and $G_{e/o}$ is $\Ge$ or $\Go$, depending on the parity of $\chi_n$. 
	
	We now consider the first sum in \eqref{Ratios}. We write the numerator using the approximate functional equation and the denominator as a Dirichlet series. The first part of the functional equation then gives
	\begin{equation}\label{First part of ratios conjecture}
		\sum_{\substack{n\leq X,\\n\equiv1\mod 4}}\mu^2(n)\sum_{m\leq x}\frac{\chi_n(m)}{m^{1/2+\alpha}}\sum_{k\geq1}\frac{\mu(k)\chi_n(k)}{k^{1/2+\beta}}.
	\end{equation} We extend the sum over $m$ all the way to infinity, and assume that the main contribution comes from the diagonal terms, which is computed in the following lemma:
	\begin{lemma}\label{Lemma character sum}
		For $b=1$ or $3$, we have 
		\begin{equation}\label{key}
			\begin{aligned}
				\sum_{\substack{n\leq X,\\ n\equiv b\mod 4}}\mu^2(n)\chi_n(\l)=\begin{cases}
					\frac{X}{2\zeta(2)}a(4\l)+\text{small}&\hbox{if $\l=\square$},\\\text{small}&\hbox{if $\l\neq\square$,}
				\end{cases}
			\end{aligned}
		\end{equation} where 
		\begin{equation}\label{key}
			a(k)=\prod_{p\mid k}\frac{p}{p+1}.
		\end{equation}
	\end{lemma}
	\begin{proof}
		Since $n$ runs over odd integers, we may replace $\l$ by $4\l$. Then we have
		\begin{equation}\label{key}
			\sum_{\substack{n\leq X,\\ n\equiv b\mod 4}}\mu^2(n)\chi_n(4\l)=\sum_{n\leq X}\mu^2(n)\chi_n(4\l)\frac{\psi_1(n)\pm\psi_{-1}(n)}{2}
		\end{equation} for an appropriate sign depending on $b$.
		By Perron's formula,
		\begin{equation}\label{key}
			\sum_{n\leq X}\mu^2(n)\chi_n(4\l)\psi(n)=\frac{1}{2\pi i}\int_{(c)}A\lz s,\leg{4\l}{\cdot}\psi\pz\frac{X^s}{s}ds,
		\end{equation} where 
		\begin{equation}\label{key}
			\begin{aligned}
				A\lz s,\leg{4\l}{\cdot}\psi\pz&=\sum_{n=1}^{\infty}\frac{\mu^2(n)\leg{4\l}{n}\psi(n)}{n^s}=\\
				%&=\prod_{p}\lz1+\frac{\leg{4\l}{p}\psi(p)}{p^s}\pz=\\
				%&=\prod_p\lz1-\frac{\leg{4\l}{p}^2\psi(p)^2}{p^{2s}}\pz\lz1-\frac{\leg{4\l}{p}\psi(p)}{p^s}\pz^{-1}=\\
				&=\frac{L\lz s,\leg{4\l}{\cdot}\psi\pz}{L\lz 2s,\leg{4\l}{\cdot}^2\psi^2\pz}.
			\end{aligned}
		\end{equation}
		$A\lz s,\leg{4\l}{\cdot}\psi\pz$ has a pole at $s=1$ only if $\l=\square$ and $\psi=\psi_1$ is the principal character, in which case we have
		\begin{equation}\label{key}
			A\lz s,\leg{4\l}{\cdot}\psi\pz=\frac{\zeta(s)}{\zeta(2s)}\prod_{p\mid 4\l}\lz1-\frac{1}{p}\pz\lz1-\frac1{p^2}\pz^{-1}=\frac{\zeta(s)}{\zeta(2s)}a(4\l).
		\end{equation} Therefore the residue at $s=1$ is $\frac{a(4\l)}{\zeta(2)},$ which gives the main term \begin{equation}\label{key}
			\frac{X}{2\zeta(2)}a(4\l).
		\end{equation}
	\end{proof}
	
	By Lemma \ref{Lemma character sum}, the main term of \eqref{First part of ratios conjecture} should be 
	\begin{equation}\label{key}
		\frac{X}{2\zeta(2)}\sum_{mk=\square}\frac{\mu(k)a(4mk)}{m^{1/2+\alpha}k^{1/2+\beta}}.
	\end{equation}
	We split the sum depending on the parity of $mk$ and expand into an Euler product:
	\begin{equation}\label{RCC}
		\begin{aligned}
			&\sum_{mk=\square}\frac{\mu(k)a(4mk)}{m^{1/2+\alpha}k^{1/2+\beta}}=\frac23\sum_{\substack{mk=\square,\\mk\odd}}\frac{\mu(k)a(mk)}{m^{1/2+\alpha}k^{1/2+\beta}}+\sum_{\substack{mk=\square,\\mk\even}}\frac{\mu(k)a(mk)}{m^{1/2+\alpha}k^{1/2+\beta}}\\
			&=\frac23\prod_{p>2}\sum_{m+k\even}\frac{\mu(p^k)a(p^{m+k})}{p^{m(1/2+\alpha)+k(1/2+\beta)}}\\
			&+\lz\sum_{\substack{m+k\even,\\m+k\geq1}}\frac{\mu(2^k)a(2^{m+k})}{2^{m(1/2+\alpha)+k(1/2+\beta)}}\pz\prod_{p>2}\sum_{m+k\even}\frac{\mu(p^k)a(p^{m+k})}{p^{m(1/2+\alpha)+k(1/2+\beta)}}\\
			&=\frac23\lz1+\sum_{\substack{m+k\even,\\m+k\geq1}}\frac{\mu(2^k)}{2^{m(1/2+\alpha)+k(1/2+\beta)}}\pz\prod_{p>2}\sum_{m+k\even}\frac{\mu(p^k)a(p^{m+k})}{p^{m(1/2+\alpha)+k(1/2+\beta)}}.
		\end{aligned}
	\end{equation} The product over $p>2$ is the same as in the ratios conjecture for fundamental discriminants, so using (2.25) in \cite{CoSn}, it equals \begin{equation}\label{Euler product evaluation}
		\frac{2}{3}\cdot\frac{\zeta_{(2)}(1+2\alpha)}{\zeta_{(2)}(1+\alpha+\beta)}\cdot P_{D,2}(\alpha,\beta),
	\end{equation} where 
	\begin{equation}\label{key}
		P_{D,2}(\alpha,\beta)=\prod_{p>2}\lz1+\frac{p^{\alpha-\beta}-1}{p^{\alpha-\beta}(p+1)(p^{1+\alpha+\beta}-1)}\pz
	\end{equation}
	The remaining factor in \eqref{RCC} is 
	\begin{equation}\label{key}
		\begin{aligned}
			&1+\sum_{\substack{m\geq1,\\m\even}}\frac{1}{2^{m(1/2+\alpha)}}-\sum_{\substack{m\geq0,\\m\odd}}\frac{1}{2^{m(1/2+\alpha)+1/2+\beta}}\\
			%&=1+\frac{1}{2^{1+2\alpha}-1}-\frac{1}{2^{1+\alpha+\beta}}\cdot\frac{1}{1-\frac1{2^{1+2\alpha}}}=\\
			&=\lz1-\frac{1}{2^{1+\alpha+\beta}}\pz\lz1-\frac1{2^{1+2\alpha}}\pz^{-1},
		\end{aligned}
	\end{equation} so it recovers the missing factors in zeta's in \eqref{Euler product evaluation}.
	Therefore the first main term for $n\equiv1\mod 4$ is
	\begin{equation}\label{key}
		\frac{X}{3\zeta(2)}\frac{\zeta(1+2\alpha)}{\zeta(1+\alpha+\beta)}P_{D,2}(\alpha,\beta).
	\end{equation}
	The second part of the functional equation contributes
	\begin{equation}\label{key}
		\Ge(1/2+\alpha)\sum_{n\leq X}\mu^2(n)\bfrac{n}{\pi}^{-\alpha}\sum_{mk=\square}\frac{\mu(k)a(4mk)}{m^{1/2-\alpha}k^{1/2+\beta}}.
	\end{equation} The inner sum is similar as above with $\alpha$ replaced by $-\alpha$, so this term gives
	\begin{equation}\label{key}
		\frac{X^{1-\alpha}}{(1-\alpha)3\zeta(2)}\cdot\frac{\pi^\alpha \Ge(1/2+\alpha)\zeta(1-2\alpha)}{\zeta(1-\alpha+\beta)}P_{D,2}(-\alpha,\beta).
	\end{equation}
	
	Finally, the computation for $n\equiv3\mod 4$ will be similar with $\Ge$ replaced by $\Go$ in the second main term, so we obtain
	\begin{conjecture} Let $-1/4<\re(\alpha)<1/4$, $\frac1{\log X}\ll\re(\beta)<1/4,$ and $\im(\alpha),\im(\beta)\ll X^{1-\epsilon}$. Then
		\begin{equation}\label{key}
			\begin{aligned}
				&\sum_{\substack{n\leq X,\\n\odd}}\frac{\mu^2(n)L(1/2+\alpha,\chi_n)}{L(1/2+\beta,\chi_n)}=\frac{2X}{3\zeta(2)}\frac{\zeta(1+2\alpha)}{\zeta(1+\alpha+\beta)}P_{D,2}(\alpha,\beta)\\
				&+\frac{X^{1-\alpha}\pi^{\alpha}\lz\Ge\lz\frac12+\alpha\pz+\Go\lz\frac12+\alpha\pz\pz\zeta(1-2\alpha)}{(1-\alpha)3\zeta(2)\zeta(1-\alpha+\beta)} P_{D,2}(-\alpha,\beta)\\
				&+O(X^{1/2+\epsilon}).
			\end{aligned}
		\end{equation}
	\end{conjecture}
	
	To obtain an asymptotic for the sum of logarithmic derivatives, we differentiate with respect to $\alpha,$ and set $\alpha=\beta=r$. 
	
	For the first term, we have
	\begin{equation}\label{key}
		\begin{aligned}
			&\frac{d}{d\alpha}\frac{\zeta(1+2\alpha)}{\zeta(1+\alpha+\beta)}P_{D,2}(\alpha,\beta)\Bigg\vert_{\alpha=\beta=r}
			%&=P_{D,2}(r,r)\lz\frac{2\zeta'(1+2r)}{\zeta(1+2r)}-\frac{\zeta(1+2r)\zeta'(1+2r)}{\zeta(1+2r)^2}\pz+P'_{D,2}(r,r)=\\
			=\frac{\zeta'(1+2r)}{\zeta(1+2r)}+P'_{D,2}(r,r),
		\end{aligned}
	\end{equation} where we noted that $P_{D,2}(r,r)=1$, and
	\begin{equation}\label{key}
		P'_{D,2}(r,r)=\sum_{p>2}\frac{\log p}{(p+1)(p^{1+2r}-1)},
	\end{equation} so the contribution of the first term is 
	\begin{equation}\label{key}
		\frac{2X}{3\zeta(2)}\lz\frac{\zeta'(1+2r)}{\zeta(1+2r)}+\sum_{p>2}\frac{\log p}{(p+1)(p^{1+2r}-1)}\pz.
	\end{equation}
	For the second term, we notice that only one term in the derivative survives due to the factor $\frac{1}{\zeta(1-\alpha+\beta)}$, so it equals
	\begin{equation}\label{key}
		-\frac{X^{1-r}\pi^{r}\lz\Ge\lz\frac12+r\pz+\Go\lz\frac12+r\pz\pz\zeta(1-2r)}{(1-r)3\zeta(2)} P_{D,2}(-r,r).
	\end{equation} We are thus led to the following conjecture:
	\begin{conjecture} For $\frac1{\log X}\ll\re(r)\ll\frac14,$ $\im(r)\ll X^{1-\epsilon}$, we have
		\begin{equation}\label{key}
			\begin{aligned}
				&\sum_{n\leq X}\frac{\mu^2(n)L'(1/2+r,\chi_n)}{L(1/2+r,\chi_n)}\\
				&=\frac{2X}{3\zeta(2)}\lz\frac{\zeta'(1+2r)}{\zeta(1+2r)}+\sum_{p>2}\frac{\log p}{(p+1)(p^{1+2r}-1)}\pz\\
				&-\frac{X^{1-r}\pi^{r}\lz\Ge\lz\frac12+r\pz+\Go\lz\frac12+r\pz\pz\zeta(1-2r)}{(1-r)3\zeta(2)} P_{D,2}(-r,r)\\
				&+O(X^{1/2+\epsilon}).
			\end{aligned}
		\end{equation}
	\end{conjecture}
	The two main terms match our result in Theorem \ref{Theorem for log derivatives for squarefree moduli}, after using $\M f(s)\approx1/s$ and noticing that 
	\begin{equation}\label{key}
		\begin{aligned}
			\frac{P_{D,2}(-r,r)}{3\zeta(2)}&=\frac{1}{{3\zeta(2)}}\prod_{p>2}\lz1+\frac{1-p^{2r}}{(p+1)(p-1)}\pz\\
			%&=\frac{1}{{3\zeta(2)}}\prod_{p>2}\frac{p^2-p^{2r}}{p^{2}-1}=\\
			%&=\frac{1}{{3\zeta(2)}}\prod_{p>2}\lz\frac{p^2-1}{p^2-p^{2r}}\pz^{-1}=\\
			&=\frac{1}{{3\zeta(2)}}\prod_{p>2}\lz1-1/p^2\pz^{-1}\lz1-p^{2r-2}\pz\\
			%&=\frac{\zeta_{(2)}(2)}{3\zeta(2)\zeta_{(2)}(2-2r)}=\\
			&=\frac{1}{4\zeta_{(2)}(2-2r)}
		\end{aligned}
	\end{equation}
	
	\section{\bf{Multivariable complex analysis}}
	A general reference for the theory of multivariable complex analysis is \cite{Hor}.
	
	\begin{defin}
		An open set $R\subset \C^n$ is a domain of holomorphy if there are no open sets $R_1,R_2\subset\C^n$ such that $\emptyset\neq R_1\subset R\cap R_2,$ $R_2$ is connected and not contained in $R$, and for any holomorphic function $f$ on $R$, there is a function $f_2$ holomorphic on $R_2$ such that $f=f_2$ on $R_1.$ 
	\end{defin}
	Open balls $B(c,r)$ centered in $c$ of radius $r$ are domains of holomorphy.
	The following is a generalization of vertical strips in $\C^n$.
	\begin{defin}
		An open set $T\subset\C^n$ is a tube if there is an open set $U\subset\R^n$ such that $T=U+i\R^n=\{z\in\C^n:\ \re(z)\in U\}.$
	\end{defin}
	
	The following is (a generalization of) Bochner's Tube Theorem \cite{Boc}.
	\begin{theorem}
		A tube domain is a domain of holomorphy if and only if it is convex.	
	\end{theorem}
	We denote the convex hull of $T$ by $\hat T$. In particular, every holomorphic function on $T$ has a holomorphic continuation to $\hat T$.
	
	The following is useful in showing that some properties of holomorphic functions extend to their analytic continuations.
	\begin{theorem}\label{Preimage is domain of holomorphy}
		Let $R_1\subset\C^m,R_2\subset\C^n$ be domains of holomorphy, and $f~:~R_1~\rightarrow~\C^n$ a holomorphic map. Then \begin{equation}\label{key}
			R=f^{-1}(R_2)=\{z\in R_1:f(z)\in R_2\}
		\end{equation} is a domain of holomorphy.
	\end{theorem}
	
	The following proposition is used to estimate the size of the meromorphic continuations of our triple Dirichlet series in vertical strips.
	\begin{prop}\label{Extending inequalities}
		Assume that $T\subset \C^n$ is a tube domain, $g,h:T\rightarrow \C$ are holomorphic functions, and let $\tilde g,\tilde h$ be their holomorphic continuation to $\hat T$. If  $\lvert g(z)\rvert \leq \lvert h(z)\rvert $ for all $z\in T$, and $h(z)$ is nonzero in $T$, then also $\lvert \tilde g(z)\rvert \leq \lvert \tilde h(z)\rvert $ for all $z\in \hat T$.
	\end{prop}
	\begin{proof}
		Since $h(z)$ is nonzero, the function $f(z)=g(z)/h(z)$ is holomorphic in $T$, and hence has a holomorphic continuation $\tilde f$ to $\hat T$. By our assumptions, $\lvert \tilde f(z)\rvert \leq 1$ for all $z\in T,$ so $T\subset \tilde f^{-1}(B(0,1))$. However, by Theorem \ref{Preimage is domain of holomorphy}, $\tilde f^{-1}(B(0,1))$ is a domain of holomorphy, so it is the whole $\hat T.$
	\end{proof}
	
	\begin{acknowledgements}
		I am very thankful to my supervisor Chantal David for her encouragement and many useful conversations. I also thank Alexandra Florea for useful coments on an earlier version of this paper.
		
		The author was supported by bourse de doctorat en recherche (B2X) of Fonds de recherche du Québec -- Nature et technologies (FRQNT).
		
		I would like to thank the anonymous referee for his careful reading of the paper and useful comments.
	\end{acknowledgements}


\begin{thebibliography}{CFKRS05} % '2nd argument contains the widest acronym'
		\bibitem[Blo11]{Blo} V. Blomer, \emph{Subconvexity for a double Dirichlet series}, Compositio Math. {147} (2011), 355-374.
		
		\bibitem[Boc38]{Boc} S. Bochner, \emph{A theorem on analytic continuation of functions in several variables}, Ann. of Math. (2) {39} (1938), 14–19.
		
		\bibitem[Bum]{Bum} D. Bump, \emph{Multiple Dirichlet series}, available online at \url{http://sporadic.stanford.edu/bump/multiple.pdf}.
		
		\bibitem[BFG12]{BFG} D. Bump, S. Friedberg, D. Goldfeld (editors), 2012, Multiple Dirichlet Series, L-functions and Automorphic Forms, vol. {300} of Progress in Mathematics. Birkhäuser/Springer, New York, 2012. 
		
		\bibitem[BFH96]{BFH} D. Bump, S. Friedberg, J. Hoffstein, \emph{On some applications of automorphic forms to number theory}, Bull. A.M.S {33} (1996), 157–175.
		
		\bibitem[BFK11]{BFK1} H. Bui, A. Florea, J. Keating, \emph{Type-I contributions to the one and two level densities of quadratic Dirichlet L-functions over function fields}, J. Number Theory 221 (2021), 389-423.
		
		\bibitem[BFK21]{BFK2}  H. M. Bui, A. Florea, J. P. Keating, \emph{The Ratios conjecture and upper bounds for negative moments of L-functions over function fields}, preprint (2021), 	arXiv:2109.10396.
		
		\bibitem[CFH06]{CFH}  G. Chinta, S. Friedberg, J. Hoffstein, Multiple Dirichlet series and automorphic forms. In Multiple Dirichlet series, automorphic forms, and analytic number theory, Proc. Sympos. Pure Math., vol. {75} (2006), 3-41, Amer. Math. Soc., Providence, RI.
		
		\bibitem[CFKRS05]{CFKRS} J. B. Conrey, D. W. Farmer, P. Keating, M. Rubinstein, N. Snaith, \emph{Integral moments of L-functions,} Proc. Londond Math. Soc. (3) 91 (2005), no. 1, 33-104.
		
		\bibitem[CFZ05]{CFZ2} J. B. Conrey, D. W. Farmer, M. R. Zirnbauer, \emph{Howe pairs, supersymmetry, and ratios of random characteristic polynomials for the unitary groups UN }, preprint, 2005. arXiv math-ph/0511024.
		
		\bibitem[CFZ08]{CFZ1} J. B. Conrey, D. W. Farmer and M. R. Zirnbauer, \emph{Autocorrelation of ratios of L-functions}, Commun. Number Theory Phys., 2 (3) (2008), 593-636.	
		
		\bibitem[CoSn07]{CoSn} J. B. Conrey, N. Snaith, \emph{Applications of the L-functions Ratios conjecture}, Proc. Lond. Math. Soc. 93 (3) (2007), 594–646.
		
		
		
		
		\bibitem[DGH03]{DGH} A. Diaconu, D. Goldfeld, J. Hoffstein, \emph{Multiple Dirichlet series and moments of zeta and L-functions}, Compositio Math. {139} 2003, 297-360.
		
		\bibitem[DHJ15]{DHJ} C. David, D. K. Huynh, J. Parks, \emph{One-level density of families of elliptic curves and the Ratios Conjecture}, Reseach in Number Theory 1 (2015), no. 1, 1-37.
		
		
		\bibitem[Far93]{Far}D. W. Farmer, \emph{Long mollifiers of the Riemann zeta-function}, Mathematika 40 (1) (1993), 71–87.
		
		\bibitem[FiMi15]{FiMi} D. Fiorilli, S. J. Miller, \emph{Surpassing the ratios conjecture in the 1-level density of Dirichlet L-functions}, Algebra Number Theory 9 (2015), no. 1, 13-52.
		
		\bibitem[Flo21]{Flo} A. Florea, \emph{Negative moments of L-functions with small shifts over function fields}, preprint (2021), arXiv:2111.10477.
		
		\bibitem[Gar]{Gar} P. Garrett, \emph{Analytic continuation, functional equation: examples}, online notes available at \url{https://www-users.cse.umn.edu/~garrett/m/mfms/notes_c/analytic_continuations.pdf}
		
		\bibitem[GoHo85]{GH} D. Goldfeld, J. Hoffstein, \emph{Eisenstein series of 1/2-integral weight and the mean value of real Dirichlet L-series}, Invent. Math. 80 (1985), 185-208.
		
		\bibitem[GJM+10]{GJM+} J. Goes, S. Jackson, S.J. Miller, D. Montague, K. Ninsuwan, R. Peckner, T. Pham, \emph{A unitary test of the L-functions Ratios Conjecture}, J. Number Theory 130 (2010), 2238–2258.
		
		\bibitem[H\"or66]{Hor} L. H\"ormander, An introduction to complex analysis in several variables, Van Nostrand, Princeton, N.J., 1966.
		
		
		\bibitem[HMM11]{HMM} D. K. Huynh, S. J. Miller, R. Morrison, \emph{An elliptic curve test of the L-functions Ratios conjecture,} J. Number Theory 131 (2011), no. 6 1117-1147.
		
		
		\bibitem[KaSa99a]{KaSa1} N. Katz and P. Sarnak, \emph{Random Matrices, Frobenius Eigenvalues and Monodromy}, AMS
		Colloquium Publications 45, AMS, Providence, 1999.
		
		\bibitem[KaSa99b]{KaSa2} N. Katz and P. Sarnak, \emph{Zeros of zeta functions and symmetries}, Bull. AMS 36, 1999, 1-26.
		
		\bibitem[Mil08]{Mil1} S.J. Miller, \emph{A symplectic test of the L-functions Ratios Conjecture}, Int. Math. Res. Not. 2008 (3) (2008), 36 pp.
		
		\bibitem[Mil09]{Mil2} S.J. Miller, \emph{An orthogonal test of the L-Functions Ratios Conjecture}, Proc. Lond. Math. Soc. (2009), doi:10.1112/plms/pdp009.
		
		\bibitem[MiMo11]{MiMo} S. J. Miller, D. Montague, \emph{An orthogonal test of the L-functions Ratios Conjecture II}, Acta Arith. 146 (2011) 53–90.
		
		
		
		\bibitem[Mon73]{Mon} H. Montgomery, \emph{The pair correlation of zeros of the zeta function}, Analytic Number Theory, Proc. Sympos. Pure Math. 24, Amer. Math. Soc., Providence, 1973, 191-193.
		
		\bibitem[\"OzSn99]{OzSn} A. E. Özlük and C. Snyder, \emph{On the distribution of the nontrivial zeros of quadratic L-functions close to the real axis}, Acta Arith. 91 (1999), 209–228.
		
		
		
		\bibitem[Sou00]{Sou} K. Soundararajan, \emph{Nonvanishing of Quadratic Dirichlet L-functions at s=1/2}, Ann. of Math. (2) 152 (2000), 447–488.
	\end{thebibliography}
\end{document}